\documentclass[12pt]{article}
\usepackage{amsmath}
\usepackage{amscd}
\usepackage{amsthm}
\usepackage{amssymb} \usepackage{latexsym}
\usepackage{eufrak}
\usepackage{euscript}
\usepackage[width=19cm, height=24cm]{geometry}
\usepackage{epsfig}
\usepackage{graphics}
\usepackage{array}
\usepackage{enumerate}
\usepackage{authblk}
\theoremstyle{theorem}
\newtheorem{thm}{Theorem}[section]
\theoremstyle{corollary}
\newtheorem{corollary}{Corollary}[section]
\theoremstyle{lemma}
\newtheorem{lemma}{Lemma}[section]
\theoremstyle{definition}

\theoremstyle{proposition}
\newtheorem{proposition}{Proposition}[section]
\theoremstyle{proof}

\theoremstyle{remark}

\def\G{\mathcal{G}}    

\newcommand{\bel}[1]{\begin{equation}\label{#1}}

\newcommand{\be}{\begin{equation}}

\newcommand{\ba}{\begin{eqnarray}}
\newcommand{\ea}{\end{eqnarray}}

\newcommand{\qe}{\end{equation}}

\begin{document}
\title{On the spectrum of hypergraphs}

\author{\rm Anirban Banerjee}

\affil{Department of Mathematics and Statistics}
\affil[ ]{Indian Institute of Science Education and Research Kolkata}
\affil[ ]{Mohanpur-741246,  India}
\affil[ ]{\textit {{\scriptsize anirban.banerjee@iiserkol.ac.in}}}

\maketitle

\begin{abstract}
Here we study the spectral properties of an underlying weighted graph of a non-uniform hypergraph by introducing different connectivity matrices, such as adjacency, Laplacian and normalized Laplacian matrices.
We  show that  different structural properties of a  hypergrpah,
 can  be well studied using spectral properties of these matrices.  
Connectivity of a hypergraph is also investigated  by the eigenvalues of these operators. 
Spectral radii of the same are bounded by the degrees of a  hypergraph. The diameter of a  hypergraph is also bounded
 by the eigenvalues of its connectivity matrices. 
 We characterize  different properties of a regular  hypergraph  characterized by the spectrum. 
 Strong (vertex) chromatic number of a hypergraph is bounded by the eigenvalues. 
Cheeger constant on a hypergraph is defined and we show that it can be bounded  by the smallest nontrivial eigenvalues of Laplacian matrix and normalized Laplacian matrix, respectively, of a connected hypergraph. 
We also show an approach to study random walk on a (non-uniform) hypergraph that can be performed by analyzing the spectrum of transition probability operator which is defined on that hypergraph. Ricci curvature on hypergraphs is introduced in two different ways. We show that if the Laplace operator, $\Delta$, on a hypergraph satisfies a curvature-dimension type inequality $CD (\mathbf{m}, \mathbf{K})$ with $\mathbf{m}>1 $ and  $\mathbf{K}>0 $ then any non-zero eigenvalue of $- \Delta$ can be bounded below by 	$ \frac{ \mathbf{m} \mathbf{K}}{ \mathbf{m} -1  } $. 
Eigenvalues of a normalized Laplacian operator defined on a connected hypergraph can be bounded by the Ollivier's Ricci curvature of the hypergraph.

\end{abstract}

\section{Introduction}

	In spectral graph theory, eigenvalues of an operator or a matrix, defined on a graph,  are investigated and different properties of the graph structure are explored from these eigenvalues. Adjacency matrix, Laplacian matrix, normalized Laplacian matrix are the popular matrices to study in spectral graph theory \cite{Brouwer2011, chung1997, Cvetkovic2009}. Depending on the graph structure, various bounds on eigenvalues have been estimated. Different relations of graph spectrum with its diameter, coloring, and connectivity have been established. Eigenvalues also play an important role to characterize graph connectivity by retraining edge boundary, vertex boundary, isoperimetric number, Cheeger constant, etc. 
	Isoperimetric problems  deal with optimal relations between the size of a cut and the size of  separated parts. 
	Similarly, Cheeger constant shows how difficult it is to cut the Riemannian manifold into two large pieces \cite{Cheeger1970}. The concept of Cheeger constant in spectral geometry has been incorporated in a very similar way in spectral graph theory. The Cheeger constant of a graph can be bounded above and below by the smallest nontrivial eigenvalue of the Laplacian matrix and normalized Laplacian matrix, respectively, of the graph \cite{chung1997, mohar1989}. 
	Ricci curvature on a graph \cite{Bauer2017,Jost2014,Lin2010} has been introduced  which is analogous to the notion of Ricci curvature in Riemannian geometry \cite{Bakry1985, Ollivier2009}. Many results have been proved on manifolds with Ricci curvature bounded below. Lower Ricci curvature bounds have been derived in the context of finite connected graphs. 
	Random walk on  graphs is also studied by defining transition probability operator on the same \cite{Grigoryan2011}. Eigenvalues of the transition probability operator can be estimated from the spectrum of normalized graph Laplacian \cite{Banerjee2008}.

	 Unlike in a graph, an edge of a hypergraph can be formed  with more than two vertices.  Thus an edge  of a hypergraph is the nonempty subset  of the vertex set of that hypergraph \cite{Voloshin}. Different aspects of a hypergraph  like, Helly property, fractional transversal number, connectivity, chromatic number have been well studied \cite{Berge1984, Voloshin2002}. 
	 A hypergraph is used to be represented by an incidence graph which is a bipartite graph with vertex classes, the vertex set and the edge set of the hypergraph and it has been exploited to study Eulerian property, the existence of different cycles, vertex and edge coloring in hypergraphs. 

	A hypergraph can also be represented by a hypermatrix which is a multidimensional array. A recent trend has been developed to explore spectral hypergraph theory using different connectivity hypermatrices. An $m$-uniform hypergraph on $n$ vertices, where each edge
	contains the same, $m$, number of vertices can  easily be
	represented by a  hypermatrix of order $m$ and dimension $n$. 
	In 2005 \cite{Qi2005}, L.~Qi  and independently L.H.~Lim \cite{Lim2005}  introduced the concept of eigenvalues of a real supersymmetric tensor (hypermatrix). Then  spectral theory for tensors started to develop. 
%
	Afterward,  many researchers  analyzed different  eigenvalues of several connectivity tensors (or hypermatrices), namely, adjacency tensor, Laplacian tensor, normalized Laplacian tensors, etc.
	 Various properties of eigenvalues of a tensor have been studied in \cite{Chang2008, Chang2009, Li2013, Ng2009, Shao2013, shaoshan2013,  Yang2010, YangYang20101}. Using characteristic polynomial, the spectrum of the adjacency matrix of a graph is extended for uniform hypergraphs in \cite{Cooper2012}. Different properties of eigenvalues of Laplacian and signless Laplacian tensors of a uniform hypergraph have been studied in  \cite{Hu2013_2, Hu2014, Hu2015, Qi2014, Qi2014_2}.  Recently  Banerjee et al. defined different hypermatrices for non-uniform hypergraphs and studied their spectral properties \cite{Banerjee2017}. 
	 Many, but not all properties of spectral graph theory could be extended to spectral hypergraph theory.
	We also refer to \cite{Qi2017}  for detailed reading on spectral analysis of hypergraphs using different tensors. 
	One of the disadvantages to study spectral hypergraph theory by tensors (hypermatrices) is the computational complexity  to compute the eigenvalues, which is NP-Hard.

On the other hand, a simple approach for studying a hypergraph is to  represent it by an underlying graph, i.e., by a 2-graph. There are many approaches, such as weighted clique expansion \cite{Rodriguez2002, Rodriguez2003, Rodriguez2009}, clique averaging \cite{AgarwalEtAl2005}, star expansion  \cite{ZienEtAl1999}.

In this article, we study the spectral properties of the underlying weighted graph of a non-uniform hypergraph by introducing different  linear operators (connectivity matrices). This underlying graph corresponding to a uniform hypergraph is similar as studied by  Rodr\'iguez, but the weights of the edges are different. We show that spectrum of these matrices (or operators) can reveal many structural properties of hypergraphs.
 Connectivity of a hypergraph is also studied   by the eigenvalues of these operators.
	 Spectral radii of the same have been bounded by the degrees of a hypergraph. We  also bound the diameter of a hypergraph by the eigenvalues of its connectivity matrices. Different properties of a regular hypergraph are characterized by the spectrum. Strong (vertex) chromatic number of a hypergraph is bounded by the eigenvalues. 
	 We   define Cheeger constant on a hypergraph and show that it can be bounded above and below by the smallest nontrivial eigenvalues of Laplacian matrix and normalized Laplacian matrix, respectively, of a connected hypergraph. 
	 We also show an approach to study random walk on a hypergraph that can be performed by analyzing the spectrum of  transition probability operator which is defined on that hypergraph.
	 Ricci curvature on hypergraphs is introduced in two different ways. We  show that if the Laplace operator, $\Delta$, on a hypergraph satisfies a curvature-dimension type inequality $CD (\mathbf{m}, \mathbf{K})$ with $\mathbf{m}>1 $ and  $\mathbf{K}>0 $ then any non-zero eigenvalue of $- \Delta$ can be bounded below by 	
	 $ \frac{ \mathbf{m} \mathbf{K}   }{ \mathbf{m} -1  } $. 
	The spectrum of normalized Laplacian operator on a connected hypergraph is also bounded by the Ollivier’s Ricci curvature of the hypergraph.
\\

\noindent
Now we recall some definitions related to hypergraphs.
	A \textit{hypergraph} $\mathcal{G}$ is a pair $\mathcal{G}= (V, E)$ where $V$ is a set of elements called vertices,  and $E$ is a set of non-empty subsets of $V$ called  edges. If all the edges of  $\G$ have the same carnality, $m$, then $\G$ is called an \textit{$m$-uniform hypergraph}.
	The \textit{rank}  $r(\G)$ and \textit{co-rank} $cr(\G)$ of a hypergraph $\G$ are the maximum and minimum, respectively, of the cardinalities of the edges in $\G$. 
 Two  vertices $i,j\in V$ are called adjacent if they  belong to an edge together, i.e., $i,j\in e$ for some $e\in E$ and it is denoted by $i \sim j$.

Let $\mathcal{G}(V,E)$ be an $m$-uniform hypergraph with $n$ vertices and  let $K^m_n$ be the complete $m$-uniform hypergraph with $n$ vertices. Further let    $\bar{\mathcal{G}}(V, \bar{E})$ be the ($m$-uniform) \textit{complement of $\mathcal{G}$} which is also an $m$-uniform hypergraph such that an edge $e\in \bar{E}$ if and only if $ e \notin E$. Thus the edge set of  $K^m_n$ is $E\cup \bar{E}$. A   hypergraph $\mathcal{G}(V, E)$ is called \textit{bipartite} if $V$ can be partitioned into two disjoint subsets $V_1$ and $V_2$ such that for each edge $e\in E$, $e\cap V_1 \ne \emptyset$ and $e\cap V_2 \ne \emptyset$.  An  $m$-uniform complete bipartite hypergraph  is denoted by $K^m_{n_1,n_2}$, where $|V_1|=n_1$ and $|V_2|=n_2$.

 	 The \textit{Cartesian product},  $\G_1 \times \G_2$,  of two hypergraphs $\G_1 = (V_1, E_1)$ and $\G_2 = (V_2, E_2)$ is defined by the vertex set $V(\G_1 \times \G_2)=V_1 \times V_2$ and the edge set  $E(\G_1 \times \G_2)= \big\{ \{a\} \times e : a \in V_1,  e \in E_2\big\}
 	\bigcup \big\{ e \times \{x\} : e \in E_1,  x \in V_2\big\}.$
Thus, the vertices $(a,x), (b,y) \in V(\G_1 \times \G_2)$ are adjacent,$(a,x)\sim (b,y)$, if and only if, either $a=b$ and $x\sim y$ in $\G_2$, or $a\sim b$ in $\G_1$ and $x=y$. Clearly, if $\G_1 $ and $ \G_2$ are two $m$-uniform hypergraphs with $n_1$ and $n_2$ vertices, respectively, then $\G_1 \times \G_2$ is also an $m$-uniform hypergraph with $n_1 n_2$ vertices.
 
For a  set $S\subset V$ of a hypergraph $\G (V,E)$, the \textit{edge boundary} $\partial S = \partial_\G S$ is the set of edges in $\G$ with vertices in both $S$ and $V\setminus S$, i.e., $\partial S = \{e \in E : i,j\in e, i\in S \text{ and } j\in   V\setminus S \}$. Similarly  the \textit{vertex boundary} $\delta S$ for $S$ to be the set of all vertices in $V\setminus S$  adjacent to some vertex in $S$, i.e., $\delta S = \{i \in V\setminus S : i,j \in e \in E, j\in S   \}$.
The \textit{Cheeger constant} (isoperimetric number) $h(\G)$ of a hypergraph $\G (V,E)$ is defined as
$$h(\G) := \inf_{\phi \ne S \subset V} \bigg\{ \frac{|\partial S|}{\min (\mu(S), \mu(V\setminus S))} \bigg\},$$ where $\mu$ is a measure on subsets of vertices. 
 Note that, depending on the choice of measure $\mu$ we use different tools. For example, if we consider equal weights $1$ for all  vertices in a subset $S$ then $\mu(S)$ becomes the number of vertices in $S$, i.e., $|S|$
 and combinatorial Laplacian is a better tool to use here. On the other hand,  if we choose the weight of a vertex  equal to its degree, then $\mu(S) = \sum_{i\in S} d_i$ and normalized Laplacian will be a better choice in this case. Moreover, sometimes for a weighted graph, we take $\mu ' (\partial S) = \sum_{e \equiv (i,j) \in \partial S} w_{ij}$ instead of $|\partial S|$ in the numerator of $h(\G)$, where $\mu ' (\partial S)$ is a measure on the  set of edges, $\partial S$, and  $w_{ij}$ is  weight of an edge $(i,j)$.


In this article we always consider finite non-uniform hypergraph $\G (V,E)$, i.e., $|V| < \infty$, if not mentioned otherwise.

\section{Adjacency matrix of hypergraphs}
The adjacency matrix $A_\mathcal{G} = [(A_\mathcal{G})_{ij}]$ of a hypergraph $\mathcal{G}=(V, E)$ is defined as 
$$(A_\mathcal{G})_{ij} = \sum_{\substack{e  \in E, \\ i,j\in e} } \frac{1}{|e|-1}.$$  For an  $m$-uniform  hypergraph the adjacency matrix becomes $(A_\mathcal{G})_{ij} = d_{ij}\frac{1}{m-1},$
 where $d_{ij}$ is the \textit{codegree} of vertices $i$ and $j$. The codegree $d_{ij}$ of vertices $i$ and $j$ is the number of edges (in $\mathcal{G}$) that contain the vertices $i$ and $j$ both, i.e., $d_{ij}= |\{ e\in E:  i,j \in e \}|$.  The above definition of adjacency matrix for a uniform  hypergraph  is similar, but not the same, as defined in \cite{Rodriguez2002}.  Now, the degree, $d_i$, of a vertex $i\in V$ which is the number of edges that contain $i$  can be expressed as $ d_i  =      \sum_{j = 1}^{n} (A_\mathcal{G})_{ij}.$ Now we explore the operator form of the adjacency matrix defined above.
Let $\mathcal{G}(V,E)$ be a hypergraph on $n$ vertices. Let us consider a real-valued function $f$ on $\mathcal{G}$, i.e., on the vertices of $\mathcal{G}$, $f: V \rightarrow \mathbb{R}$. The set of such functions forms a vector space (or a real inner product space) which is isomorphic to $\mathbb{R}^n$. For such two functions $f_1$ and $f_2$  on $\mathcal{G}$ we take their inner product 
$\langle f_1,f_2 \rangle = \sum_{i\in V} f_1(i)f_2(i).$ This  inner product space is also  isomorphic to $\mathbb{R}^n$.
Let us choose a basis $\mathcal{B} =\{g_1,g_2,\dots,g_n \}$ such that $g_i(j)=\delta_{ij}$. Now we find the  adjacency operator $T$ such that $[T]_\mathcal{B} = A_\G$. Here, we also denote $T$ by $A_\G$.
Now, our  adjacency operator $A_\mathcal{G}$ (which is a linear operator) is defined as
$$ (A_\mathcal{G} f)(i) = \sum_{j, i\sim j} \sum_{\substack{e  \in E, \\ i,j\in e} } \frac{1}{|e|-1}  f(j).$$ 
It is easy to verify that
$\langle A_\mathcal{G} f_1,f_2 \rangle = \langle  f_1, A_\mathcal{G} f_2 \rangle
,$ for all $f_1, f_2 \in \mathbb{R}^n$, i.e., the operator $A_\mathcal{G}$ is symmetric w.r.t.~$\langle .,. \rangle$. So the eigenvalues of $A_\mathcal{G}$ are real. 
Now onwards we shall use the operator and the matrix form of $A_\mathcal{G}$ interchangeably.

 Clearly, for a connected hypergraph $\mathcal{G}$ the  adjacency matrix $A_\mathcal{G}$, which is real and non-negative  (i.e., all of its entries are a non-negative real number), possesses a Perron eigenvalue with positive real eigenvector. Moreover, for an undirected hypergraph $A_\mathcal{G}$ is symmetric.
The hypergraph $\mathcal{G}$, the corresponding weighted graph ${G}[A_\mathcal{G}]$ (constructed from the  adjacency matrix $A_\mathcal{G}$ of $\mathcal{G}$) and the graph ${G}_0[A_\mathcal{G}]$ have the similar property regarding graph connectivity and coloring. Here ${G}_0[A_\mathcal{G}]$ is the underlying unweighted graph of ${G}[A_\mathcal{G}]$. So, 
   $\mathcal{G}$ is connected if and only if ${G}[A_\mathcal{G}]$ is connected. Thus it is easy to show that a hypergraph $\mathcal{G}$ is connected if and only if the highest eigenvalue of $A_\mathcal{G}$ is simple and possesses a positive eigenvector (see Cor.~1.3.8, \cite{Cvetkovic2009}).
  We can also estimate the upper bound for  spectral radius $\rho(A_\mathcal{G}) $ for a connected hypergraph $\mathcal{G}(V,E)$ as
  $\rho(A_\mathcal{G}) \le  \max_{\substack{i,j,\\ i \sim j}}\{\sqrt{d_id_j}\}.$
  This equality holds if and only if $\mathcal{G}$ is a regular hypergraph (see Cor 2.5,\cite{Das2008}).

\begin{thm}\label{Adj:Thm:GraphCrossProduct}
Let $\G_1 =(V_1, E_1)$and $ \G_2= (V_2, E_2)$ be two $m$-uniform hypergraphs on $n_1$ and $n_2$ vertices, respectively. If $\lambda$ and $\mu$ are eigenvalues of $A_{\G_1} $and $A_{\G_2}$, respectively, then $\lambda + \mu$ is an eigenvalue of $A_{\G_1 \times \G_2} $.
\end{thm}
\begin{proof}
	Let $a,b \in V_1$ and $x, y\in V_2$ be any vertices of $\G_1$ and $\G_2$, respectively. 
Let $\alpha$ and $\beta$ be the eigenvectors corresponding to the eigenvalues $\lambda$ and $\mu$, respectively.
Let $\gamma \in \mathbb{C}^{n_1n_2}$ be a vector  with the entries $\gamma(a,x) = \alpha(a) \beta(x)$, where $(a,x) \in [n_1] \times [n_2]$. Now we show that $\gamma$ is an eigenvector of $A_{\G_1 \times \G_2}$ corresponding to the eigenvalue $\lambda + \mu$. Thus,
\begin{eqnarray*}
	 \sum_{(b,y), (a, x)\sim (b,y)}  \big(A_{\G_1 \times \G_2}\big)_{(a,x), (b,y)} \gamma (b,y) 
&	= & \sum_{(b,y), (a, x)\sim (b,y)} \frac{d_{(a,x) (b,y)}}{m-1} \gamma (b,y) \\
	{ } & = & 	\sum_{(b,x),  (a, x)\sim (b,x)} \frac{d_{ab} }{m-1} \alpha(b) \beta(x)  +	\sum_{(a,y), (a, x)\sim (a,y)} \frac{ d_{xy}}{m-1} \alpha(a) \beta(y)\\
	{ } & = & 	\beta(x) \lambda \alpha(a)   + \alpha(a)\mu \beta(x)	\\
	{ } & = &  (\lambda + \mu) \gamma(a,x).
\end{eqnarray*}	
Hence the proof follows.
\end{proof}
The $n$-dimensional $m$-uniform \textit{cube hypergraph} $Q(n,m)$ consists of the vertex set $V=\{0,1,2,\dots,m-1\}^{n}$ 
and the edge set $E=\{\{x_1,x_2,\dots,x_m\}: x_i \in V \text{ and } x_i,x_j \text{\;differs exactly in one coordinate when $i\neq{j}$}\}$ \cite{BC1996}.
Note that $Q(1,m)=K^{m}_{m}$ and $Q(1,m)\times{Q(1,m)}=Q(2,m)$. In general $$Q(n,m)=\overbrace{Q(1,m)\times{Q(1,m)\times\dots\times{Q(1,m)}}}^{n}.$$
Since the  eigenvalues of $A_{K^{m}_{m}}$ are $\{1,\overbrace{\dfrac{-1}{m-1},\dfrac{-1}{m-1},\dots,\dfrac{-1}{m-1}}^{m-2}\}$, using the above theorem we can easily  find the  eigenvalues of $A_{Q(n,m)}$.

\subsection{Diameter of a hypergraph and eigenvalues of adjacency matrix} 

A \textit{path} $v_0-v_1$ of length $l$ between two vertices $v_0,v_1\in V$ in a hypergraph $\mathcal{G}(V,E)$ is an alternating sequence $v_0e_1v_1e_2v_2\dots v_{l-1}e_lv_l$ of distinct vertices $v_0,v_1,v_2,\dots, v_l$ and distinct edges $e_1,e_2,\dots, e_l$, such that, $v_{i-1},v_i \in e_i$ for $i=1,\dots,l$.
The  \textit{distance}, $d(i,j)$,  between two vertices $i,j$ in a hypergraph $\mathcal{G}$ is the  minimum length of a $i-j$ path.  
The \textit{diameter}, $diam(\mathcal{G})$, of a hypergraph $\mathcal{G}(V,E)$ is the maximum distance between any pair of vertices in $\mathcal{G}$, i.e.,
$diam(\mathcal{G}) = \max \{d(i,j): i,j\in V \}.$ Now it is easy to show that the diameter of a hypergraph $\mathcal{G}$ is less than the number of distinct eigenvalues of $A_\mathcal{G}$.

\begin{thm}\label{Adj:diam:2ndEig} 
 L et $\mathcal{G}(V,E)$ be a connected hypergraph with $n$ vertices and minimum edge carnality 3. Let $\theta$  be the second largest eigenvalue (in absolute value) of $A_\mathcal{G}$. Then
 $$ diam(\mathcal{G}) \le \bigg\lfloor 1 + \frac{\log((1-\alpha^2)/\alpha^2)}{\log(\lambda_{max}/\theta)} \bigg\rfloor,$$ where $\lambda_{max}$ is the largest eigenvalue of  $A_\mathcal{G}$ with the unit eigenvector $X_1=((X_1)_1, (X_1)_2,\dots, (X_1)_n)^t$ and $\alpha= \min_i \{(X_1)_i$\}.
\end{thm}
\begin{proof}
$A_\mathcal{G}$  is real symmetric and thus have orthonormal eigenvectors $X_l$ with $A_\mathcal{G} X_l = \lambda_l X_l$, where $\lambda_1=\lambda_{max}$.
 Let us choose $i,j\in V$ such that $d(i,j)=diam(\mathcal{G})$ and $r\ge diam(\mathcal{G})$ be  a positive integer. We 
  try to find the minimum value of $r$ such that $(A^r_\mathcal{G})_{ij}>0$.
  Using spectral decomposition  of  $A_\mathcal{G}$, $(A^r_\mathcal{G})_{ij}$ can be express as
 \begin{align*}  
     (A^r_\mathcal{G})_{ij}
     & =   \sum_{l=1}^{n} \lambda_l^r (X_l X^t_l)_{ij}\\  
     & \ge \lambda_{max}^r (X_1)_i (X_1)_{j} - \bigg|  \sum_{l=2}^{n} \lambda_l^r (X_l)_i (X_l)_{j} \bigg| \\
     & \ge \alpha^2 \lambda_{max}^r - \theta^r \bigg (\sum_{l=2}^{n} |(X_l)_i|^2 \bigg)^{1/2} \bigg (\sum_{l=2}^{n} |(X_l)_j|^2 \bigg)^{1/2}\\
     & \ge \alpha^2 \lambda_{max}^r - \theta^r (1- \alpha^2).
\end{align*}
 Now, $(A^r_\mathcal{G})_{ij}>0$ if $(\lambda_{max}/\theta)^r > (1- \alpha^2)/\alpha^2 $, which implies that 
 $r > \frac{\log((1-\alpha^2)/\alpha^2)}{\log(\lambda_{max}/\theta)}  .$ Thus the proof follows.
 \end{proof}
\begin{corollary}
 Let $\mathcal{G}(V,E)$ be a  $k$-regular connected hypergraph with $n$ vertices. Let $\theta$  be the second largest eigenvalue (in absolute value) of $A_\mathcal{G}$. Then
 $$ diam(\mathcal{G}) \le \bigg\lfloor 1 + \frac{\log(n-1)}{\log(k/\theta)} \bigg\rfloor.$$ 
\end{corollary}
 \noindent
\textbf{Remark: }For graphs, the above bounds are more sharp \cite{chung1989}. Also  note that, in Theorem \ref{Adj:diam:2ndEig}, $ \theta \ne \lambda_{max}$ since the underlying graph is  not bipartite.

\subsection{Uniform regular hypergraphs and eigenvalues  of adjacency matrices}
 Let $\mathcal{G}$ be an $m$-uniform hypergraph with  $n$ vertices and let $d_i$ be the degree of the vertex $i$ in $\mathcal{G}$. Further, let $\bar{d}_i$ be the degree of $i$ in $\bar{\mathcal{G}}$. Thus $d_i + \bar{d}_i = {n-1 \choose m-1 }$. Then $A_\mathcal{G} + A_{\bar{\mathcal{G}}} = A_{K^m_n} = \theta (J_n - I_n)$, where  $\theta = \frac{{n-2 \choose m-2}}{m-1}$, $J_n$ is the $(n \times n)$ matrix with all the entries are $1$ and $I_n$ is the $(n \times n)$ identity matrix.
 %
If $\mathcal{G}$ be an $m$-uniform $k$-regular hypergraph with  $n$ vertices then $\bar{\mathcal{G}}$ is $\big({n-1 \choose m-1 }-k\big)$-regular and $A_{\bar{\mathcal{G}}}$ is a symmetric non-negative matrix. Thus $A_{\bar{\mathcal{G}}}$ contains a Perron eigenvalue ${n-1 \choose m-1} - k$ with an  eigenvector $\textbf{\textit{1}}_{n}$. For every  non-Perron  eigenvector  $X$ of $A_\mathcal{G}$ with  an eigenvalue $\lambda$ we have $A_{\bar{\mathcal{G}}} X  
=  (-\theta -\lambda)X.$ Thus  $A_\mathcal{G}$ and $A_{\bar{\mathcal{G}}}$ have the same eigenvectors. Now we have the following proposition. 
%

\begin{proposition}
 If $\mathcal{G}$ is an $m$-uniform $k$-regular hypergraph with  $n$ vertices, then the minimum eigenvalue $\lambda_n$ of $A_\mathcal{G}$ satisfies the inequality $\lambda_n \ge k - \theta - {n-1 \choose m-1}$.
\end{proposition}
\begin{proof}
Let us order the eigenvalues of $A_\mathcal{G}$ as $k = \lambda_1 \ge \lambda_2 \ge \dots \ge \lambda_n$. Then the eigenvalues of $A_{\bar{\mathcal{G}}}$ can be ordered as ${n-1 \choose m-1} - k \ge -\theta - \lambda_n \ge \dots \ge - \theta - \lambda_2$, which implies the result.
\end{proof}


\subsection{Hypergraph coloring and adjacency eigenvalues}
 A \textit{strong vertex coloring} of a hypergraph $\mathcal{G}$ is a coloring where any two adjacent vertices get different colors. The \textit{strong (vertex) chromatic number} $\gamma(\mathcal{G})$ of a hypergraph $\mathcal{G}$ is the minimum number of colors needed to have a strong vertex coloring of $\mathcal{G}$. 
 



\begin{thm}\label{HGraphColoring:Thm1}
Let $\gamma(\mathcal{G})$ be a hypergraph with $r(\G)=r$. Then
 $$\gamma(\mathcal{G}) \le 1 +  (r-1) \lambda_{max}(A_\mathcal{G}).$$
\end{thm}
\begin{proof}
For a simple unweighted graph $G$,  the vertex chromatic number
$\chi(G) \le 1 + \lambda_{max}(G),$
where $\lambda_{max}(G)$ is the maximum eigenvalue of the adjacency matrix of $G$ \cite{Wilf1967}.
Thus  we have
$ \gamma(\mathcal{G}) =  
 \chi({{G}_0[A_\mathcal{G}]}) \le 1 + \lambda_{max}({G}_0[A_\mathcal{G}]). $
Since each element of the adjacency matrix of ${G}_0[A_\mathcal{G}]$ is less than or equals to  the same of $(r-1)A_\mathcal{G}$, we have 
$ \lambda_{max}({G}_0[A_\mathcal{G}]) \le  (r-1) \lambda_{max}(A_\mathcal{G}).$
Now the proof follows from the last two inequalities.
\end{proof}

%

\begin{thm} \label{HGraphColoring:Thm2}
 Let $\mathcal{G}$ be a hypergraph with at least one edge. Then
 $$ \gamma(\mathcal{G}) \ge 1 - \frac{\lambda_{max}(A_\mathcal{G})}{\lambda_{min}(A_\mathcal{G})}
 .$$
\end{thm}
\begin{proof}
 Let $k$ be $\gamma(\mathcal{G})$. Now  $A_\mathcal{G}$ can be partitioned as 
\begin{displaymath}
\mathbf{A_\mathcal{G}} =
\left[ \begin{array}{cccc}
0 & A_{\mathcal{G}_{12}} & \ldots  & A_{\mathcal{G}_{1k}}\\
A_{\mathcal{G}_{21}} & 0 & \ldots & A_{\mathcal{G}_{2k}}\\
\vdots & \vdots & \ddots & \vdots\\
A_{\mathcal{G}_{k1}} & A_{\mathcal{G}_{k2}} & \ldots & 0
\end{array} \right].
\end{displaymath}
Using the   Lemma 3.22 in \cite{bapat2010}  we have  $\lambda_{max}(A_\mathcal{G}) + (k-1) \lambda_{min}(A_\mathcal{G}) \le 0.$
Since $\mathcal{G}$ has at least one edge, $\lambda_{min}(A_\mathcal{G}) < 0$. Hence the proof follows.
\end{proof}

 From Theorems \ref{HGraphColoring:Thm1} and \ref{HGraphColoring:Thm2}  we have the following corollary.
 \begin{corollary}
 Let $\mathcal{G}$ be a hypergraph with at least one edge. Then
 $$|\lambda_{min}(A_\mathcal{G})| \ge 1/(r(\G)-1) .$$
 \end{corollary}

\section{Combinatorial Laplacian matrix and operator of a hypergraph}

Now we define our (combinatorial) Laplacian operator $L_\G$ for a  hypergraph $\G (V,E)$ on $n$ vertices. 
We take the same usual inner product $\langle f_1,f_2 \rangle = \sum_{i\in V} f_1(i)f_2(i)$ for the $n$ dimensional Hilbert space $L^2(\G)$ constructed with all real-valued functions $f$ on $\G$, i.e., $f: V \rightarrow \mathbb{R}$.  Now our Laplacian operator
$$ L_\G : L^2(\G)  \rightarrow L^2(\G) $$
 defined as 
$$ (L_\G f)(i) :=  \sum_{j, i\sim j} \sum_{\substack{e  \in E, \\ i,j\in e} } \frac{1}{|e|-1}(f(i)-f(j)) = d_i f(i) - \sum_{j, i\sim j} \sum_{\substack{e  \in E, \\ i,j\in e} } \frac{1}{|e|-1} f(j).$$
\noindent
It is easy to verify that $L_\G$ is symmetric (self-adjoint) w.r.t.~the usual inner product $\langle .,. \rangle$, i.e., $\langle L_\mathcal{G} f_1,f_2 \rangle = \langle  f_1, L_\mathcal{G} f_2 \rangle ,$ for all $f_1, f_2 \in \mathbb{R}^n$. So the eigenvalues of $L_\mathcal{G}$ are real. Since
$\langle L_\mathcal{G} f,f \rangle =  \sum_{ i\sim j} \sum_{\substack{e  \in E, \\ i,j\in e} } \frac{1}{|e|-1} (f(i) - f(j))^2 \ge 0,$ for all $f \in \mathbb{R}^n$, $L_\G$ is nonnegative, i.e., the eigenvalues of $L_\G$ are nonnegative.
The Rayleigh Quotient $\mathcal{R}_{L_\G} (f)$  of a function $f: V\rightarrow \mathbb{R}$ is defined as 
$$\mathcal{R}_{L_\G} (f) = \frac{\langle L_\mathcal{G} f,f \rangle}{\langle f,f \rangle} = 
\frac{    \sum_{ i\sim j} \sum_{\substack{e  \in E, \\ i,j\in e} } \frac{1}{|e|-1} (f(i) - f(j))^2  }{\sum_{i\in V} f(i)^2}.$$
For standard basis  we get the matrix from Laplacian operator $L_\G$ as
\begin{displaymath}
(L_\G)_{ij} = \left\{ \begin{array}{ll}
 d_i & \textrm{if $i=j$,}\\
 \sum_{\substack{e  \in E, \\ i,j\in e} } \frac{-1}{|e|-1} & \textrm{if $i\sim j$,}\\
 0 & \textrm{elsewhere.}
  \end{array} \right.
\end{displaymath}
So, $L_\G = D_\G - A_\G$, where $D_\G$ is the diagonal matrix where the entries are the degrees $d_i$ of the vertices $i$ of $\G$. 
Any $\lambda(L_\G) \in \mathbb{R}$ becomes an eigenvalue of $L_\G$ if, for a nonzero $u\in \mathbb{R}^n$, it satisfies the  equation
\begin{equation}\label{Lap:EigValEqu}
L_\G u =  \lambda(L_\G)u.
\end{equation}

Let us order the eigenvalues of $L_\G$ as $\lambda_1(L_\G) \le \lambda_2(L_\G) \le \dots \le \lambda_n(L_\G)$. Now find an orthonormal basis of $L^2(\G)$ consisting of eigenfunctions of $L_\G$,
$u_k, k=1,\dots,n$ as follows. 
First we find $u_1$ from the expression
$\lambda_1 (L_\G) = \inf_{u\in L^2(\G) } \{\langle L_\G u, u  \rangle : ||u||=1 \} .$
Now iteratively define Hilbert space of all real-valued functions on $\G$ with the scalar product $\langle., .\rangle$,
$ H_k := \{ v \in L^2(\G) : \langle v, u_l\rangle =0 \text{ for }l\le k \}.$ Then we start with the function $u_1$ (eigenfunction for the eigenvalue $\lambda_1(L_\G)$) and find all the eigenvalues of $L_\G$ as
$\lambda_k(L_\G) = \inf_{u\in H_k - \{0\}} \bigg \{ \frac{\langle L_\G u, u\rangle}{\langle u, u\rangle} \bigg \}.$
Thus, 
$\lambda_2(L_\G) = \inf_{0\ne u \perp u_1  } \bigg \{ \frac{\langle L_\G u, u\rangle}{\langle u, u\rangle} \bigg \}.$
We can also find  $\lambda_n (L_\G)$ as
$\lambda_n (L_\G) = \sup_{u\in L^2(\G)} \{\langle L_\G u, u  \rangle : ||u||=1 \} .$

Ger{\v s}gorin circle theorem \cite{Ger}
provides  a trivial  bounds on    eigenvalues $\lambda$  of  $L_\mathcal{G}$ as   $|\lambda|\le 2d_{max}$.
As in adjacency matrix for an $m$-uniform hypergraph $\mathcal{G}$ with $n$ vertices we have $L_\mathcal{G} + L_{\bar{\mathcal{G}}} = L_{K^m_n} =  \phi_m (n) I_n -  \theta J_n $, where  
$\phi_m (n) = \frac{n}{n-1}   {n-1 \choose m-1}$,
$\theta = \frac{{n-2 \choose m-2}}{m-1}$, $J_n$ is the $(n \times n)$ matrix with all the entries  $1$ and $I_n$ is the $(n \times n)$ identity matrix.
%
Now it is easy to show that for an $m$-uniform  hypergraph $\mathcal{G}$  with  $n$ vertices, if $0=\lambda_1 \le \lambda_2 \le \dots \le \lambda_n$ be the eigenvalues of $L_\G$ with the corresponding eigenvectors,  $\textbf{\textit{1}}_{n} = X_1, X_2, \dots, X_n$, respectively, then the eigenvalues of $L_{\bar{\mathcal{G}}}$ are $0=\lambda_1,  \phi_m (n)  - \lambda_2, \dots,  \phi_m (n)  - \lambda_n$ with the same set of corresponding eigenvectors   $ X_1, X_2, \dots, X_n$, respectively.
%
Now we  have the following computations on eigenvalues. 
	 The eigenvalues of $L_{K^m_n}$ are $0$ and $\phi_m(n)$ with the (algebraic) multiplicity  $1$ and $n-1$, respectively.
		The eigenvalues of $L_{K^m_{n_1,n_2}}$ are $0, \phi_m(n_1 + n_2), \phi_m(n_1 + n_2) - \phi_m(n_1)$ and $\phi_m(n_1 + n_2) - \phi_m(n_2)$ with the multiplicity $1, 1, n_1 -1$ and $n_2 - 1$, respectively.
		If $0, \lambda_2, \dots, \lambda_{n_1}$ and $0, \mu_2, \dots, \mu_{n_2}$ are the eigenvalues of $L_{\G_1}$ and $L_{\G_2}$, respectively, where $\G_1$ and $\G_2$ are two $m$-uniform hypergraphs  with number of vertices, $n_1$ and $n_2$, respectively, then the eigenvalues of $L_{\G_1 + \G_2}$ are $0, \phi_m(n_1 + n_2)$, $ \phi_m(n_1 + n_2) - \phi_m(n_1) + \lambda_2$, $\dots,$ $ \phi_m(n_1 + n_2) - \phi_m(n_1) + \lambda_{n_1}$, $ \phi_m(n_1 + n_2) - \phi_m(n_2) + \mu_2$, $\dots,$ $ \phi_m(n_1 + n_2) - \phi_m(n_2) + \mu_{n_2}$, where $\G_1 + \G_2 = \overline{ \bar{\G_1} \cup \bar{\G_2}}$.	 
Also note that the Theorem \ref{Adj:Thm:GraphCrossProduct}  holds for Laplacian matrices of two uniform hypergraphs and using it we can compute the  eigenvalues of $L_{Q(n,m)}$.
%

The definition of Laplacian matrix for a uniform hypergraph is similar, but not the same as defined in \cite{Rodriguez2002, Rodriguez2003, Rodriguez2009} for studying spectral properties of uniform hypergraphs. In \cite{Rodriguez2002},  Rodr\'iguez studied the Laplacian eigenvalues of a uniform hypergraph and several metric parameters such as the diameter, mean distance, excess, cutsets and bandwidth. A very trivial upper bound for diameter by distinct eigenvalues of Laplacian matrix is mentioned.  The bounds on parameters related to distance in (uniform) hypergraphs, such as eccentricity, excess were investigated in \cite{Rodriguez2003} . In \cite{Rodriguez2009}, the distance between two vertices in a (uniform) hypergraph was explained by the degree of a real polynomial of Laplacian matrix. The relation between parameters related to partition-problems of a (uniform) hypergraph with 
second smallest and the largest eigenvalues of   Laplacian matrix was explored. Only the lower bound for isoperimetric number and bipartition width were figured out in terms of the second smallest eigenvalue of the same. Upper bound for max cut, domination number and independence number was mentioned by the largest Laplacian-eigenvalue.

\subsection{Hypergraph connectivity and eigenvalues of a (combinatorial) Laplacian matrix}
Now it is easy to verify that  $\G$ is connected \textit{iff} $\lambda_2(L_\G) \ne 0$ and then 
 any constant function $u\in \mathbb{R}^n$ becomes the eigenfunction with the eigenvalue $\lambda_1(L_\G) = 0 $. If   $\G(V,E)$ has $k$ connected components, then
 the (algebraic) multiplicity of the eigenvalue $0$ of $L_\G$ is exactly $k$. So, we call $\lambda_2(L_\G)$  \textit{algebraic weak connectivity} of hypergraph $\G$. The following theorems show more relation of $\lambda_2(L_\G)$ with the different aspects of connectivity of hypergraph $\G$.

 A set of vertices in a hypergraph $\G$ is a \textit{weak vertex cut} of $\G$ if weak deletion of the vertices from that set increases the number of connected components in $\G$. 
 The \textit{weak connectivity number} $\kappa_W(\G)$ is the minimum size of a weak vertex cut in $\G$.
\begin{thm}
Let $\G(V,E)$ be a  connected hypergraph with $n (\ge 3)$ vertices, such that, $\G$ contains at least one pair of nonadjacent vertices and $d_{max} \le cr(\G)$. Then $\lambda_2(L_\G) \le \kappa_W(\G)$.
\end{thm}
\begin{proof}
 Let $W$ be a weak vertex cut of $\G$ such that $|W|=\kappa_W(\G)$. Let us partition the vertex set of $\G$ as $V_1\cup W \cup V_2$ such that no vertex in $V_1$ is adjacent to a vertex in $V_2$. Let $|V_1|=n_1$ and $|V_2|=n_2$. 
 
 Since $\G$ is connected, $u_1$ is constant. Let us construct a real-valued function $u$,    orthogonal to $ u_1$, as
 \begin{displaymath}
u(i) = \left\{ \begin{array}{ll}
 n_2 & \textrm{if $i\in V_1$,}\\
  0 & \textrm{if $i\in W$,}\\
   -n_1 & \textrm{if $i\in V_2$.}
  \end{array} \right.
\end{displaymath}
Since any vertex $j\in W$ is adjacent to a vertex in $V_1$  and also to a vertex in $V_2$, and $d_{max} \le cr(\G)$, then $d_{ij}\le cr(\G)-1$ for all $j\in W$ and $i \notin W$. 
Now,  for any vertex $i \notin W$, we define $k_i = \sum_{j\in W, j\sim i}\sum_{\substack{e  \in E, \\ i,j\in e} } \frac{1}{|e|-1} \le \sum_{j\in W, j\sim i}\frac{d_{ij}}{cr(\G)-1} \le\kappa_W(\G) $.
Thus, for all $i \in V_1$,
 $ (L_\G u)(i)  = d_i u(i) - \sum_{j\in V_1, j\sim i} \sum_{\substack{e  \in E, \\ i,j\in e} } \frac{1}{|e|-1} u(j)
 = n_2 d_i - n_2 (d_i - k_i) =n_2 k_i.$
Similarly, for all $i \in V_2$, we have 
 $(L_\G u)(i)  =  - n_1 k_i $.
Hence,
$\lambda_2(L_\G) ||u||^2 \le \langle L_\G u, u  \rangle  \le n_1 n_2^2 \kappa_W(\G) + n_1^2 n_2 \kappa_W(\G) = \kappa_W(\G)  ||u||^2
.$
\end{proof}

\begin{thm}
 Let $\G (V,E)$ be a   hypergraph on $n$ vertices. Then, for a nonempty $S\subset V$, we have
 $$ (r(\G)-1)
 \frac{\lambda_n(L_\G) |S| |V\setminus S|}{n} \ge
 |\partial S| \ge \frac{cr(\G)-1}{\lfloor
 r(\G)^2/4 \rfloor} \frac{\lambda_2(L_\G) |S| |V\setminus S|}{n}.$$
\end{thm}
\begin{proof}
 Let us  construct a real-valued function $u$, orthogonal to $ u_1$, as
 \begin{displaymath}
u(i) = \left\{ \begin{array}{ll}
 n - n_S & \textrm{if $i\in S$,}\\
  -n_S & \textrm{if $i\in V\setminus S$,}
  \end{array} \right.
\end{displaymath}
where   $n_S=|S|$.
Now we have
$$|\partial S| \frac{1}{r(\G)-1}  n^2
\le
 \langle L_\G u, u  \rangle = \sum_{ i\sim j} \Big(\sum_{\substack{e  \in E, \\ i,j\in e} } \frac{1}{|e|-1}(u(i) - u(j))^2 \Big)  
  \le |\partial S| \frac{1}{cr(\G)-1} \lfloor  r(\G)^2/4 \rfloor n^2. 
$$
The inequality on the right side holds because if $i\in S$ and $j \in V\setminus S$, then the number of terms in the parentheses in the above equation is  maximum  when $|e|=r(\G)$ and there are equal number of vertices in $e$ from $S$ and $V\setminus S$, respectively, and  is equal to  $\lfloor  r(\G)^2/4 \rfloor$. Similary, the inequality on the left side holds when there is only one vertex from a partition in $e\in \partial S$.
Thus,
$ \lambda_2(L_\G) ||u||^2 = \lambda_2(L_\G) nn_S (n-n_S) \le  \langle L_\G u, u  \rangle \le  |\partial S| \frac{1}{cr(\G)-1} \lfloor  r(\G)^2/4 \rfloor n^2 .$
Similarly, we have $ \lambda_n(L_\G) ||u||^2 = \lambda_n(L_\G) nn_S (n-n_S) \ge  \langle L_\G u, u  \rangle \ge |\partial S| \frac{1}{r(\G)-1}  n^2 $.
Hence the proof follows.
\end{proof}

Now we bound the Cheeger constant   
$$h(\G) := \min_{\phi \ne S \subset V} \bigg\{ \frac{|\partial S|}{\min (|S|, |V\setminus S|)} \bigg\}$$ of a hypergraph $\G (V,E)$ from below and above through $\lambda_2(L_\G)$.

\begin{thm}\label{CheegConstant:Lap:Thm1}
 Let $\G (V,E)$ be a connected hypergraph with $n$ vertices. Then $$h(\G) \ge 2 \lambda_2(L_\G) \Big(\frac{cr(\G)-1}{r(\G)\big(r(\G)-1\big)}\Big) .$$ 
\end{thm}
\begin{proof}
 Let $ S$ be a nonempty subset of $V$, such that, $h(\G) = |\partial S| / |S|$ and $|S| \le |V|$. Let us define a real-valued function $u$
  as
  \begin{displaymath}
u(i) = \left\{ \begin{array}{ll}
 \frac{1}{\sqrt{|S|}} & \textrm{if $i\in S$,}\\
  0 & \textrm{elsewhere.}
  \end{array} \right.
\end{displaymath}
Let us define $t_S (e) = |\{v : v\in e \cap S, e\in E \}|$ and $t(S) = \frac{\sum_{e \in \partial S} t_S(e)}{|\partial S|}$.
Now we have
 $
\frac{\langle L_\G u, u  \rangle}{||u||^2}  
 = \sum_{ e  \in E} \sum_{\substack{i\sim j,\\ \{i,j \}\in e}} \frac{1}{|e|-1} (u(i) - u(j))^2
  \le \sum_{ e  \in \partial S} \sum_{i\in e \cap S} \frac{r(\G)-1}{cr(\G)-1} u(i)^2
  = \frac{r(\G)-1}{cr(\G)-1} \sum_{ e  \in \partial S} \frac{t_S(e)}{|S|}.
$
 Thus we have 
 $ \lambda_2(L_\G) \le \frac{r(\G)-1}{cr(\G)-1}  \frac{|\partial S| t(S)}{|S|} $
 and similarly
 $\lambda_2(L_\G) \le \frac{r(\G)-1}{cr(\G)-1}  \frac{|\partial \bar{S} | t (\bar{S} ) }{| \bar{S} |},$
where $\bar{S} = V\setminus S$. Now from these two inequalities we have 
$
   \lambda_2(L_\G) \le   \frac{r(\G)\big(r(\G)-1\big)}{2(cr(\G)-1)} h(\G).
$
Thus the proof follows.
\end{proof}

\begin{thm}\label{CheegConstant:Lap:Thm2}
 Let $\G (V,E)$ be a connected hypergraph on $n$ vertices. If $d_{max}$ is the maximum degree of $\G$ and $\lambda_2 = \lambda_2(L_\G)$ then $$h(\G) < (r(\G)-1) \sqrt{(2d_{max} - \lambda_2)\lambda_2}.$$ 
\end{thm}
\begin{proof}
Let $u_2$ be the eigenfunction with the eigenvalue $\lambda_2$, such that, $||u_2||=1$. Let $ \phi \ne S \subset V$, such that $h(\G)= \frac{|\partial S|}{|S|}$ and $|S|\le |V|/2$.
 Let $u: V \rightarrow \mathbb{R}$ be a function defined by
\begin{displaymath}
u(i) = |u_2(i)|  \textrm{ for all $i\in V$.}\\
\end{displaymath}

Thus,
\begin{eqnarray}
  \lambda_2   >     \sum_{i\sim j}  \sum_{\substack{e  \in E, \\ i,j\in e} } \frac{1}{|e|-1} (u(i) - u(j))^2   \ge \frac{1}{r(\G)-1} \sum_{i\sim j} d_{ij} (u(i) - u(j))^2 =: M \text{ (say)}.\label{cheegerUBound:eq1}
\end{eqnarray}
The rest of the proof is similar to the proof for graphs  \cite{mohar1989}.
Equation (\ref{cheegerUBound:eq1}), 
by using Cauchay-Schwarz inequality,
 implies that
\begin{eqnarray}
M & 
 \ge & \frac{1}{r(\G)-1} \frac{\Big(\sum_{i\sim j} d_{ij} |u^2(i) - u^2(j)|\Big)^2 }{\sum_{i\sim j} d_{ij} (u(i) + u(j))^2}. \label{cheegerUBound:eq2}
\end{eqnarray}
Now proceed in a similar way as in the proof given in \cite{mohar1989}.
Let $ t_0 < t_1 < \dots < t_h$ be all different values of $u(i)$, $i\in V$. For  $k = 0,1, \dots, h$, let us define $V_k := \{i\in V: u(i) \ge t_k \}$, and  we denote  $\delta_k(e) = \min\{|V_k \cap e|, |(V\setminus V_k) \cap e| \}$ for each edge $ e\in \partial V_k$ and $\delta(V_k) =\min_{e\in \partial V_k} \{\delta_k(e) \}$. Let $\delta(\G) = \min_{k\in [h]} \{  \delta(V_k)\}$, where $[h]= \{1,2,\dots, h\} $.
Now
\begin{eqnarray}
 \sum_{i\sim j} d_{ij} |u^2(i) - u^2(j)| & = &  \sum_{\substack{ i\sim j\\ u(i) \ge u(j)}} d_{ij} \big ( u^2(i) - u^2(j)\big)  \nonumber \\
 & = & \sum_{k=1}^{h} \sum_{\substack{i\sim j \\ u(i) = t_k\\ u(j)=t_l<t_k}} d_{ij} \big(t_k^2 - t_{k-1}^2 + t_{k-1}^2 - \dots - t_{l+1}^2 + t_{l+1}^2 - t_{l}^2 \big)     \nonumber \\ 
  & \ge &  \sum_{k=1}^{h} \delta (V_k) |\partial V_k| (t_k^2 - t_{k-1}^2) \nonumber \\ 
  & \ge & \delta (\G) h(\G)    \sum_{k=1}^{h} |V_k| (t_k^2 - t_{k-1}^2) \nonumber \\ 
  & = & \delta (\G) h(\G)    \sum_{k=1}^{n} u(k)^2.  \label{cheegerUBound:eq3}
\end{eqnarray}
On the other hand,
\begin{eqnarray}
 \sum_{i\sim j} d_{ij} (u(i) + u(j))^2 & = &   2\sum_{i\sim j} d_{ij} \big(u(i)^2 + u(j)^2\big) -  \sum_{i\sim j} d_{ij} \big(u(i) - u(j)\big)^2\nonumber \\
   & \le & 2 d_{max} (r(\G)-1) \sum_{i=1}^n u(i)^2 -  \sum_{i=1}^n u(i)^2 \sum_{i\sim j} d_{ij} \big(u(i) - u(j)\big)^2\nonumber \\ 
      & = & ( 2 d_{max} - M) (r(\G)-1) \sum_{i=1}^n u(i)^2 \text{, using Equation (\ref{cheegerUBound:eq1})} \label{cheegerUBound:eq4}.
\end{eqnarray}
Now, from Equations (\ref{cheegerUBound:eq2}), (\ref{cheegerUBound:eq3}) and (\ref{cheegerUBound:eq4}), we get
$$
M  >  \frac{\frac{1}{r(\G)-1}  \delta (\G)^2 h(\G)^2    \big(\sum_{k=1}^{n} u(k)^2\big)^2  }{( 2 d_{max} - M) (r(\G)-1) \sum_{i=1}^n u(i)^2} 
\ge   \frac{1}{(r(\G)-1)^2}   \frac{ h(\G)^2 }{( 2 d_{max} - M)} \text{, since $\delta (\G)^2 \ge 1$}$$

Hence, $\lambda_2 > \frac{ h(\G)^2 }{(r(\G)-1)^2}   \frac{1}{( 2 d_{max} - M)} \Rightarrow h(\G) < (r(\G)-1) \sqrt{(2d_{max} - \lambda_2)\lambda_2}.$
\end{proof}

\subsection{Diameter and eigenvalues of Laplacian matrix of a hypergraph}
%
\begin{thm}\label{Lap:Diam:LBound}
For a connected hypergraph $\G (V,E)$ on $n$ vertices,
$$diam(\G) \ge \frac{4}{n(r(\G)-1) \lambda_2(L_\G)}.$$
\end{thm}
\begin{proof}
Let us consider the eigenfunction $u_2$ with the eigenvalue  $\lambda_2(L_\G)$.  Then we have
 \begin{eqnarray*} 
\lambda_2(L_\G)  & = &  \frac{  \sum_{i\sim j} \sum_{\substack{e  \in E, \\ i,j\in e} } \frac{1}{|e|-1} (u_2(i) - u_2(j))^2 }{ ||u_2||^2 }\\
  { } & \ge & \frac{1}{r(\G)-1} \lambda_2(L_{G_0[A_\mathcal{G}]})\\
 { } & \ge &     \frac{4}{n(r(\G)-1)diam({G}_0[A_\mathcal{G}])},   \text{ (using Theorem 4.2 in \cite{mohar1991}) }.
\end{eqnarray*}
Since  $diam(\G) = diam({G}_0[A_\mathcal{G}])$, thus the proof follows.
\end{proof}
 Distance $d(V_1, V_2)$ between two nonempty proper subsets, $V_1$ and $V_2$, of the vertex set of a hypergraph 
 is defined as $$ d(V_1, V_2) := \min \{d(i,j): i\in V_1, j\in V_2 \} .$$
\begin{lemma}
Let $M$ denote an $n \times n$ matrix with rows and columns indexed by the vertices of a graph $G$ and let $M(i,j)=0$ if $i,j$ are not adjacent. Now, if for some integer $t$ and some polynomial $p_t(x)$ of degree $t$, $(p_t(M))_{ij}\ne 0$ for any $i$ and $j$, then $diam(G)\le t$.
\end{lemma}

\begin{thm}\label{Lap:Diam:UBound}
Let $\G (V,E)$ be an  connected hypergraph on $n$ vertices  with at least one pair of nonadjacent vertices. Then, for $V_1, V_2 \subset V$ such that $ V_2 \ne V_1 \ne V\setminus V_2$, we have
$$d(V_1, V_2) \le \Bigg\lceil  \frac{\log\sqrt{\frac{(n-|V_1|)(n-|V_2|)}{|V_1||V_2|}}}{\log\Big(\frac{\lambda_n(L_\G) + \lambda_2(L_\G)}{\lambda_n(L_\G)- \lambda_2(L_\G)}\Big)}     \Bigg   \rceil.$$
\end{thm}
\begin{proof}
 For $V_1\subset V$, let us construct a function $u_{V_1}$ as
  \begin{displaymath}
u_{V_1}(i) = \left\{ \begin{array}{ll}
1/\sqrt{|V_1|} & \textrm{if $i\in V_1$,}\\
  0 & \textrm{elsewhere.}
  \end{array} \right.
\end{displaymath}
Let $f_i$ be orthonormal eigenfunctions of $L_\G$, such that $L_\G f_i = \lambda_i (L_\G)f_i $, for $i = 1,\dots, n$. Then we have $u_{V_1} = \sum_{i=1}^{n}  a_i f_i.$
Let us take $f_1 = \frac{1}{\sqrt{n}} \mathbf{1}_n$. Then $a_1 =\langle u_{V_1}, f_1  \rangle = \sqrt{\frac{|V_1|}{n}}.$
Similarly, for  $V_2\subset V$, we construct a function $u_{V_2} = \sum_{i=1}^{n}  b_i f_i,$ where $b_1 = \sqrt{\frac{|V_2|}{n}}.$

Now, choose a polynomial $p_t(x)= \Big( 1 - \frac{2x}{\lambda_2(L_\G) + \lambda_n(L_\G)} \Big)^t$. Clearly 
$|p_t (\lambda_i(L_\G))| \le (1-\lambda)^t ,$ for all $i = 2, 3, \dots, n,$ where $\lambda = \frac{2 \lambda_2(L_\G)}{\lambda_2(L_\G) +  \lambda_n(L_\G)  }$.
If $ \langle u_{V_2}, p_t(L_\G) u_{V_1} \rangle >0 $ for some $t$, then there is a path of length  at most $t$ between a vertex in $V_1$ and a vertex in $V_2$.
Thus,
\begin{eqnarray}
 \langle u_{V_2}, p_t(L_\G) u_{V_1} \rangle  & = &  p_t (0) a_1 b_1 + \sum_{i=2}^{n} p_t (\lambda_i(L_\G)) a_i b_i   \nonumber \\
 & \ge &  \frac{\sqrt{|V_1||V_2|}}{n}  - \bigg| \sum_{i=2}^{n} p_t (\lambda_i(L_\G)) a_i b_i  \bigg  |  \nonumber \\ 
  & \ge & \frac{\sqrt{|V_1||V_2|}}{n} - (1-\lambda)^t \sqrt{\sum_{i=2}^{n}a_i^2   \sum_{i=2}^{n}b_i^2  }  \label{DiamUBound:eq0} \\
 & = &  \frac{\sqrt{|V_1||V_2|}}{n} - (1-\lambda)^t \frac{\sqrt{(n-|V_1|) (n-|V_2|)}}{n}.   
  \label{DiamUBound:eq1}
\end{eqnarray}
The inequality (\ref{DiamUBound:eq0}) follows from Cauchy-Schwarz inequality, whereas the  equality (\ref{DiamUBound:eq1}) holds since
 $\sum_{i=2}^{n} a_i^2 = ||u_{V_1}||^2 - (|\langle u_{V_1}, f_1   \rangle |)^2 = \frac{n - |V_1|}{n}
 ,$
 and  $\sum_{i=2}^{n} b_i^2 =  \frac{n - |V_2|}{n}
 .$
 
The inequality in (\ref{DiamUBound:eq0}) is strict. This is because the equality in Cauchy-Schwarz inequality
holds if and only if $a_i = c b_i$, for all $i$, for some constant $c$. However, it is  possible only when $V_1 = V_2$ or $V_1 = V \setminus V_2$, which is not the case here. So, we get
$ \langle u_{V_2}, p_t(L_\G) u_{V_1} \rangle >  \frac{\sqrt{|V_1||V_2|}}{n} - (1-\lambda)^t \frac{\sqrt{(n-|V_1|) (n-|V_2|)}}{n}
.$
Now,  if we choose $$t \ge    \frac{\log\sqrt{\frac{(n-|V_1|)(n-|V_2|)}{|V_1||V_2|}}}{\log\Big(\frac{\lambda_n(L_\G) + \lambda_2(L_\G)}{\lambda_n(L_\G)- \lambda_2(L_\G)}\Big)},$$ $ \langle u_{V_2}, p_t(L_\G) u_{V_1} \rangle$  becomes strictly positive.
  Thus the proof follows.
\end{proof}

\begin{corollary}
	For a connected hypergraph $\G (V,E)$ on $n$ vertices and with at least one pair of nonadjacent vertices, 
	$$diam(\G) \le	\Bigg\lceil   \frac{\log(n-1)}{\log\Big(\frac{\lambda_n(L_\G) + \lambda_2(L_\G)}{\lambda_n(L_\G)- \lambda_2(L_\G)}\Big)}     \Bigg   \rceil.$$ 
\end{corollary}

\begin{corollary}
	Let $\G (V,E)$ be an  connected hypergraph  on $n$ vertices and with at least one pair of nonadjacent vertices. Then, for any $S \subset V$, we have
	$$\frac{|\delta S|}{|S|} \ge (n-|S|) \frac{1 - (1-\lambda)^2}{(1-\lambda)^2 (n-|S|) + |S|} ,$$
	where $\lambda = \frac{2 \lambda_2(L_\G)}{\lambda_n(L_\G) + \lambda_2(L_\G)}$ and $\delta S = \{j\in V\setminus S:d(i,j)=1, \text{ for some } i\in S \}$ is the vertex boundary of $S$. Moreover, if $|S| \le n/2$ then we have
	$$\frac{|\delta S|}{|S|} \ge \frac{2\lambda_n(L_\G) \lambda_2(L_\G)}{\lambda_n(L_\G)^2 + \lambda_2(L_\G)^2} .$$
\end{corollary}
\begin{proof}
Take $V_1 = S$, $V_2 = V \setminus S - \delta S$ and $t=1$. Then, using  Equation (\ref{DiamUBound:eq1}) of Theorem \ref{Lap:Diam:UBound}, we have 
$$
	0   =  \langle u_{V_2}, p_{t=1}(L_\G) u_{V_1} \rangle 
>    \frac{\sqrt{|V_1||V_2|}}{n} - (1-\lambda) \frac{\sqrt{(n-|V_1|) (n-|V_2|)}}{n}.
$$
Since $V \setminus V_2 = S \cup \delta S$, this implies 
$(1-\lambda)^2 (n- |S|) (|S| + |\delta S|) > |S| (n - |S| - |\delta S|).$
Thus the first part of the result follows. 

\noindent
Now, when $|S| \le n - |S|$, from the above inequality we get
 $$
\frac{|\delta S|}{|S|}  \ge  \frac{1 - (1-\lambda)^2}{(1-\lambda)^2  + |S|/ (n-|S|)}
\ge 	\frac{1 - (1-\lambda)^2}{(1-\lambda)^2  + 1} 	
=  \frac{2\lambda_n(L_\G) \lambda_2(L_\G)}{\lambda_n(L_\G)^2 + \lambda_2(L_\G)^2} .
$$\end{proof}

\subsection{Bounds on $\lambda_2(L_\G)$ and $\lambda_n(L_\G)$  of a  hypergraph $\G$}
Let $e\equiv \{i_1, i_2, \dots, i_m \} \in E$ be any edge in $\G$. Then, for $e$ and $u\in \mathbb{R}^n$, we define a  homogeneous polynomial of degree $2$  in $n$ variables by
$$L_\G(e)u^2 := \sum_{j=1}^{|e|}u(i_j)^2 - \frac{1}{|e|-1} \sum_{\substack{j=1,\\i_j\in e}}^{|e|} \sum_{\substack{k=1,\\i_j\ne i_k\in e}}^{|e|} u(i_j) u(i_k).$$
Thus, $\langle L_\G u, u \rangle = \sum_{e\in E} L_\G(e)u^2 .$
\begin{thm}
 Let $\G (V,E)$ be a connected hypergraph on $n(>2)$ vertices. Then 
 $$ \lambda_2(L_\G) \le \min \bigg\{ \frac{d_{i_1} + d_{i_2} + \dots + d_{i_{|e|}} - |e|}{|e|} : e \equiv \{i_1, i_2, \dots, i_{|e|} \} \in E \bigg\} .$$
\end{thm}
\begin{proof}
Let $e \equiv \{i_1, i_2, \dots, i_{|e|} \} \in E$ be any edge in $\G$. Now,  let us construct a function $u\in \mathbb{R}^n$, as
 \begin{displaymath}
u(i) = \left\{ \begin{array}{ll}
{|e|}^{-1/2} & \textrm{if $i\in e$,}\\
  0 & \textrm{elsewhere.}
  \end{array} \right.
\end{displaymath}
Now,
$
\lambda_2(L_\G)   \le   \sum_{e\in E} L_\G(e)u^2
   =   \frac{1}{|e|}(d_{i_1} + d_{i_2} + \dots + d_{i_{|e|}}) - \frac{1}{|e|-1}\cdot 2 {|e| \choose 2} \cdot \frac{1}{|e|}.
$
Thus the proof follows.
\end{proof}
\begin{corollary}
 Let $\G (V,E)$ be a  connected hypergraph with $n(>2)$ vertices. Then 
 $$ \lambda_n(L_\G) \ge \min \bigg\{ \frac{d_{i_1} + d_{i_2} + \dots + d_{i_{|e|}} - |e|}{|e|} : e\equiv\{i_1, i_2, \dots, i_{|e|} \} \in E \bigg\} .$$
\end{corollary}

\begin{thm}\label{Lap:Eig:UBound}
	Let $\G (V,E)$ be an $m(>2)$-uniform connected hypergraph with $n$ vertices. Then 
	$$\lambda_n(L_{\G}) \le \max \Bigg\{ \frac{2d_i (m-1) -1 + 
		\sqrt{4(m-1)^2 d_i m_i  D_{max}^2 - 2d_i (m-1) + 1}
		}	{2(m-1)} : i \in V \Bigg\} 	,$$
		where $m_i = (\sum_{j, j\sim i} d_j)/d_i(m-1)$ is the average $2$-degree of the vertex $i$ and $D_{max}= \max\{ d_{xy}: x,y\in V  \}$.
\end{thm}	
\begin{proof}
Let $u_n$ be an eigenfunction of $L_{\G}$ with the eigenvalue $\lambda_n(L_\G)$ ($ = \lambda_n(\G)$, say). Then we have
\begin{equation}  \label{UBound:eq1}
\sum_{i,j, i\sim j} d_{ij} (u_n(i) - u_n(j))^2 = (m-1) \lambda_n(\G) \sum_{i=1}^{n} u_n(i)^2.
\end{equation}
From the eigenvalue equation  of $L_\G$ (i.e., from (\ref{Lap:EigValEqu}))  for the vertex $i$  we have
$
(m-1) (d_i - \lambda_n(L_{\G})) u_n(i) \le  \sum_{j, j\sim i} D_{max} u_n(j).
$
Using Lagrange identity and  summing both sides over $i$  we get
\begin{equation} \label{UBound:eq2}
\sum_{i=1}^{n}  (m-1)^2 (d_i - \lambda_n(L_{\G}))^2 u_n(i)^2    \le  (m-1) \sum_{i=1}^{n}  d_i D_{max}^2 	 \sum_{j, j\sim i} u_n(j)^2 
- \sum_{i=1}^{n}  \sum_{ \substack{  1\le j < k \le n \\ j\sim i, k\sim i}} D_{max}^2  (u_n(j) - u_n(k))^2
\end{equation}
Now, since
\begin{equation*}
  \sum_{i=1}^{n}  d_i  \sum_{j, j\sim i} u_n(j)^2  =  (m-1)  \sum_{i=1}^{n}  d_i m_i  u_n(i)^2,
\end{equation*}
and 
\begin{eqnarray*}
\sum_{i=1}^{n}  \sum_{ \substack{  1\le j < k \le n \\ j\sim i, k\sim i}} D_{max}^2  (u_n(j) - u_n(k))^2 & \ge &
\sum_{i,j, i\sim j} d_{ij} (u_n(i) - u_n(j))^2\\
{ } & = & (m-1) \lambda_n(\G) \sum_{i=1}^{n}  u_n(i)^2 \text{ (using (\ref{UBound:eq1})), } 
\end{eqnarray*}
 Equation (\ref{UBound:eq2}) becomes
\begin{equation}\label{UBound:eq3}
 \sum_{i=1}^{n} \big[(m-1)^2 (d_i - \lambda_n(\G))^2  - (m-1)^2 d_i m_i D_{max}^2 + (m-1) \lambda_n(\G)  \big] u_n(i)^2 \le 0.
\end{equation}
Hence there exists a vertex $i$ for which we have
$
 (m-1) (d_i - \lambda_n(\G))^2  - (m-1) d_i m_i D_{max}^2 +   \lambda_n(\G)  \le 0.
$
This provides our desired result.
\end{proof}

\begin{corollary}
		Let $\G (V,E)$ be an $m(>2)$-uniform connected hypergraph on $n$ vertices. Then 
		$$\lambda_n(L_{\G}) \le   \frac{2d_{max} (m-1) -1 + 
			\sqrt{4(m-1)^2 d_{max}^2  |E|^2 - 2d_{min} (m-1) + 1}
		}	{2(m-1)},$$
		where $d_{max}$ and $d_{min}$ are the maximum and the minimum degrees, respectively, of $\G$.
\end{corollary}
\begin{proof}
	Since $d_im_i = (\sum_{j, j\sim i} d_j)/(m-1) \le d_{max}^2$ and $D_{max} \le |E|$ the result follows from the above theorem.
\end{proof}

%
Another upper bound of $\lambda_n(L_{\G})$ can also be found from the Theorem 5 in \cite{ROJO2007} as
$
\lambda_n(L_{\G}) \le \frac{1}{2} \max_{i\sim j} \Big\{d_i + d_j + \frac{1}{m-1} \Big(\sum_{k\sim i, k\nsim j}d_{ik} + \sum_{k\sim j, k\nsim i}d_{jk} +  \sum_{k\sim i, k\sim j} |d_{ik} - d_{jk} | \Big) \Big \}
.$

\section{Normalized Laplacian matrix and operator of a hypergraph}
Now we define normalized Laplacian operator and matrix $\Delta_\G$ for a  hypergraph $\G (V, E)$ on  $n$ vertices. Let $\mu$ be a natural measure on $V$ given by $\mu (i) = d_i$. We consider the inner product 
$\langle f_1,f_2 \rangle_{\mu} := \sum_{i\in V} \mu(i) f_1(i)f_2(i),$ for the $n$ dimensional Hilbert space $l^2(V, \mu)$, given by 
$ l^2(V, \mu) = \{f: V \rightarrow \mathbb{R} \}.$  Now our normalized Laplacian operator
$ \Delta_\G : l^2(V, \mu) \rightarrow l^2(V, \mu) $
 is  defined as 
\begin{equation}\label{NorLap:Def1}
(\Delta_\G f)(i) :=  \frac{1}{d_i}\sum_{j, i\sim j} \sum_{\substack{e  \in E, \\ i,j\in e} } \frac{1}{|e|-1}(f(i)-f(j)) =  f(i) -  \frac{1}{d_i}\sum_{j, i\sim j} \sum_{\substack{e  \in E, \\ i,j\in e} } \frac{1}{|e|-1} f(j).
\end{equation}
It is easy to verify that the eigenvalues of $\Delta_\G$ are real and nonnegative, since 
$ \langle  f_1, \Delta_\mathcal{G} f_2 \rangle_\mu = \langle \Delta_\mathcal{G} f_1,f_2 \rangle_\mu ,$ for all $f_1, f_2 \in \mathbb{R}^n$ and 
$\langle \Delta_\mathcal{G} f,f \rangle_\mu = \frac{1}{m-1} \sum_{ i\sim j} d_{ij} (f(i) - f(j))^2 \ge 0,$ for all $f \in \mathbb{R}^n$. 
The Rayleigh Quotient $\mathcal{R}_{\Delta_\G} (f)$  of a function $f: V\rightarrow \mathbb{R}$ is defined as 
\begin{equation}\label{NorLap:Def:RayQuo1}
\mathcal{R}_{\Delta_\G} (f) = \frac{\langle \Delta_\mathcal{G} f,f \rangle_\mu}{\langle f,f \rangle_\mu} = 
\frac{    \sum_{ i\sim j} \sum_{\substack{e  \in E, \\ i,j\in e} } \frac{1}{|e|-1} (f(i) - f(j))^2  }{\sum_{i\in V} d_if(i)^2}.
\end{equation}
For standard basis  we get the matrix form of normalized Laplacian operator $\Delta_\G$ as
\begin{equation}\label{NorLapMatrix:Def1}
(\Delta_\G)_{ij} = \left\{ \begin{array}{ll}
 1 & \textrm{if $i=j$,}\\
 \frac{-1}{d_i}\sum_{\substack{e  \in E, \\ i,j\in e} } \frac{1}{|e|-1} & \textrm{if $i\sim j$,}\\
 0 & \textrm{elsewhere.}
  \end{array} \right.
\end{equation}
So, $\Delta_\G = I - R_\G$, where $R_\G = A_\G D_\G^{-1}$ is  normalized adjacency matrix, which is a row-stochastic matrix. $R_\G$ can be considered as a
probability transition matrix of a random walk on   $\G$.

Now we order the eigenvalues of $\Delta_\G$ as $\lambda_1(\Delta_\G) \le \lambda_2(\Delta_\G) \le \dots \le \lambda_n(\Delta_\G)$ and find an orthonormal basis of $l^2(V, \mu)$ consisting of eigenfunctions of $\Delta_\G$, $u_k, k=1,\dots,n$, as we did it for Laplacian operator. 
The expression
$\lambda_1 (\Delta_\G) = \inf_{u\in l^2(V, \mu)  - \{ 0\}} \{\langle \Delta_\G u, u  \rangle_\mu : ||u||=1 \} $ provides $u_1$ and $\lambda_1(\Delta_\G)$. The rest of the eigenvalues are iteratively estimated from the expression
$\lambda_k(\Delta_\G) = \inf_{u\in \mathcal{H}_k - \{0\}} \bigg \{ \frac{\langle \Delta_\G u, u\rangle_\mu}{\langle u, u\rangle_\mu} \bigg \},$ where
$ \mathcal{H}_k := \{ v \in l^2(V, \mu) : \langle v, u_l\rangle_\mu =0 \text{ for }l\le k \}.$
$\lambda_2(\Delta_\G)$ can also be expressed as
$\lambda_2(\Delta_\G)  = \inf_{u \perp u_1} \bigg \{ \frac{\langle \Delta_\G u, u\rangle_\mu}{\langle u, u\rangle_\mu} \bigg \}
.$

We can also define  normalized Laplacian operator (and matrix) on a  hypergraph $\G (V,E)$ on $n$ vertices as follows. Here, we consider the usual inner product $\langle f_1,f_2 \rangle := \sum_{i\in V}  f_1(i)f_2(i),$ for the $n$ dimensional Hilbert space $L^2(V)$ constructed with all real-valued functions $f: V \rightarrow \mathbb{R}$ and  the other normalized Laplacian operator
$ \mathcal{L}_\G : L^2(V)\rightarrow L^2(V) $
and  is  defined as 
\begin{equation}\label{NorLap:Def2}
(\mathcal{L}_\G f)(i) :=  \frac{1}{\sqrt{d_id_j}}\sum_{j, i\sim j} \sum_{\substack{e  \in E, \\ i,j\in e} } \frac{1}{|e|-1}(f(i)-f(j)) =  f(i) -  \frac{1}{\sqrt{d_id_j}}\sum_{j, i\sim j} \sum_{\substack{e  \in E, \\ i,j\in e} } \frac{1}{|e|-1}f(j).
\end{equation}
For standard basis  we get the matrix form of the above normalized Laplacian operator $\mathcal{L}_\G$ as
\begin{equation}\label{NorLapMatrix:Def2}
(\mathcal{L}_\G)_{ij} = \left\{ \begin{array}{ll}
 1 & \textrm{if $i=j$,}\\
  \frac{-1}{\sqrt{d_id_j}}\sum_{\substack{e  \in E, \\ i,j\in e} } \frac{1}{|e|-1} & \textrm{if $i\sim j$,}\\
 0 & \textrm{elsewhere.}
  \end{array} \right.
\end{equation}
Two normalized Laplacian operators in (\ref{NorLap:Def1}) and (\ref{NorLap:Def2}) are equivalent. Hence, the matrices in (\ref{NorLapMatrix:Def1}) and (\ref{NorLapMatrix:Def2}) are similar and thus have the same spectrum. In this article we use the normalized Laplacian operator defined in (\ref{NorLap:Def1}) and its matrix form\footnote{
Note that,
two non-isomorphic hypergraphs of order $m > 2$ may have the same normalized Laplacian matrix $\Delta_\mathcal{G}$ (or the normalized adjacency matrix $R_{\G}$). 
For an example, It happens when all the 2-element subsets of the vertex set of the hypergraph are subsets of a fixed number of edges. For instance, existence of a  (combinatorial) simple incomplete 2-design on the vertex set of a hypergraph where each edge is considered as a block. A particular example is the Fano plane, which is a finite projective plane of order $2$ with $7$ points, represents  a $3$-uniform hypergraph on 7 vertices with the vertex set $\{ 1, 2, 3, 4, 5, 6, 7\}$ and the edge set $\{ \{1,2,3 \}, \{ 1,4,7\}, \{1,5,6 \},$ $\{2,4,6 \},\{2,5,7 \},\{3,4,5 \},\{3,6,7 \}\}$, where each pair of vertices belongs to exactly one edge.  Fano plane is a regular balanced incomplete  block $(7,3,1)$-design. Thus, the normalized Laplacian (adjacency) matrices for Fano plane and $K_7^3$, respectively, are the same.} in (\ref{NorLapMatrix:Def2}).
It is easy to verify that the eigenvalues of $\Delta_{\G}$ for an $m(>2)$-unform hypergraph lie in $[0,2)$ and the number of connected components in $\G$ is equal to the (algebraic) multiplicity of eigenvalue $0$. When $\G$ is connected, $u_1$ is constant.
Many theorems for normalized Laplacian matrix  can be constructed  similar to the theorems for Laplacian matrix (operator). We see that
$\lambda_2(\Delta_{\G})$ can also bound the Cheeger constant defined as,
$$h(\G) := \inf_{\phi \ne S \subset V} \bigg\{ \frac{|\partial S|}{\min (vol(S),  vol(\bar{S}))} \bigg\}$$ for a hypergraph $\G (V,E)$ from below and above. Here ${vol } (S) = \sum_{i\in S} d_i$.
This Cheeger constant $h(\G)$ can also be bounded above and bellow by $\lambda_2(\Delta_\G)$, respectively, as follows.
\begin{thm}
 Let $\G (V,E)$ be a  connected hypergraph on $n$ vertices. Then $$ 2 \lambda_2(L_\G) \Big(\frac{cr(\G)-1}{r(\G)\big(r(\G)-1\big)}\Big)
   \le h(\G) < (r(\G)-1) \sqrt{(2 - \lambda_2)\lambda_2}.$$  
\end{thm}
\begin{proof}
The proof is similar to the proofs of Theorem \ref{CheegConstant:Lap:Thm1} and Theorem \ref{CheegConstant:Lap:Thm2}.
\end{proof}

%
%
%
A similar thoerem as Theorem \ref{Lap:Diam:LBound} can be written for $\lambda_2(\Delta_\G)$ as
\begin{thm}
For a connected hypergraph $\G$ on $n$ vertices,
$$diam(\G) \ge \frac{4}{n(r(\G)-1) d_{max} \lambda_2(\Delta_\G)}.$$
\end{thm}
Theorem \ref{Lap:Diam:UBound}  also holds for the respective eigenvalues of $\Delta_\G$ as

\begin{thm}
Let $\G (V,E)$ be a connected hypergraph with $n$ vertices and at least one pair of nonadjacent vertices. Then, for $V_1, V_2 \subset V$ such that $ V_2 \ne V_1 \ne V\setminus V_2$, we have
$$d(V_1, V_2) \le \Bigg\lceil  \frac{\log\sqrt{\frac{vol(\bar{V_1})vol(\bar{V_2})}{vol(V_1) vol(V_2)}}}{\log\Big(\frac{\lambda_n(\Delta_\G) + \lambda_2(\Delta_\G)}{\lambda_n(\Delta_\G)- \lambda_2(\Delta_\G)}\Big)}     \Bigg   \rceil.$$
\end{thm}
Then we have the similar corollary as we have for $L_\G$.
\begin{corollary}
	For a connected  hypergraph $\G (V,E)$ with $n$ vertices and at least one pair of nonadjacent vertices, 
	$$diam(\G) \le	\Bigg\lceil   \frac{\log(\frac{(n-1)d_{max}}{d_{min}})}{\log\Big(\frac{\lambda_n(\Delta_\G) + \lambda_2(\Delta_\G)}{\lambda_n(\Delta_\G)- \lambda_2(\Delta_\G)}\Big)}     \Bigg   \rceil.$$ Moreover, if $\G $ is regular then 
	$$diam(\G) \le	\Bigg\lceil   \frac{\log(n-1)}{\log\Big(\frac{\lambda_n(\Delta_\G) + \lambda_2(\Delta_\G)}{\lambda_n(\Delta_\G)- \lambda_2(\Delta_\G)}\Big)}     \Bigg   \rceil.$$ 
\end{corollary}
Now we find bounds on eigenvalues of $\delta_\G$. It is easy to verify that for a hypergraph $\G$ containing at least one pair of nonadjacent vertices  $\lambda_2(\Delta_{\G}) \le 1 \le \lambda_n(\Delta_{\G}) $.
%
A similar theorem of Theorem \ref{Lap:Eig:UBound} can be stated as follows.
\begin{thm}
	Let $\G (V,E)$ be an $m$-uniform connected hypergraph with $n$ vertices. Then 
	$$\lambda_n(\Delta_{\G}) \le \max \Bigg\{ \frac{2 (m-1)d_i -1 + 
		\sqrt{ 1  - 4(m-1)d_i +  4(m-1)^2 d_i  D_{max}^2 m_i  }
		}	{2(m-1)d_i} : i \in V \Bigg\} 	,$$
		where $m_i = (\sum_{j, j\sim i} d_j)/d_i(m-1)$ is the average $2$-degree of the vertex $i$ and $D_{max}= \max\{ d_{xy}: x,y\in V  \}$.
\end{thm}	
\begin{proof}
Let $u_n$ be an eigenfunction of $\Delta_{\G}$ with the eigenvalue $\lambda_n(\Delta_\G)$. Then, from the eigenvalue equation for  $\lambda_n(\Delta_\G)$ and $u_n$,  we have
\begin{equation} 
\sum_{i,j, i\sim j} d_{ij} (u_n(i) - u_n(j))^2 = (m-1) \lambda_n(\Delta_\G) \sum_{i=1}^{n} d_i u_n(i)^2.
\end{equation}
Rest of the proof is similar to the proof of Theorem \ref{Lap:Eig:UBound}.
Now the expression in (\ref{UBound:eq3}) becomes
 $
 \sum_{i=1}^{n} \big[(m-1) (1 - \lambda_n(\Delta_\G)^2)d_i  - (m-1) m_i D_{max}^2 +  \lambda_n(\Delta_\G)  \big] u_n(i)^2 \le 0.
$
Using similar arguments as in Theorem \ref{Lap:Eig:UBound} we have
$
 (m-1) d_i  \lambda_n(\Delta_\G)^2 
 + (1 - 2 (m-1) d_i  )  \lambda_n(\Delta_\G)
+ (m-1)(d_i -m_i D_{max}^2)   \le 0,
$ 
which proves the theorem.
\end{proof}
When $m=2$, $\G$ becomes a \textit{triangulation} and the above upper bound coincides with the 
 result  proved  in \cite{Li2014} for a triangulation. Now we have the following corollary.
\begin{corollary}
		For an $m$-uniform connected hypergraph $\G (V,E)$ with the maximum and the  minimum degrees $d_{max}$ and $d_{min}$, respectively, we have
	$$\lambda_n(\Delta_{\G}) \le   \frac{2 (m-1)d_{max} -1 + 
				\sqrt{ 1  - 4(m-1)d_{min} +  4(m-1)^2  d_{max}^2  |E|^2  }
				}	{2(m-1)d_{min}}.$$
\end{corollary}

\section{Random walk on hypergraphs}
A random walk on a hypergraph $\G(V,E)$ can be considered as a sequence of vertices $v_0, v_1, \dots, v_t$ and it can be determined by the transition probabilities $P(u,v) = Prob(x_{i+1} = v | x_i = u)$ which is independent of $i$. Thus, a simple random walk on a  hypergraph $\G(V,E)$ is a \textit{Markov chain}, where a \textit{Markov kernel} on $V$ is a function
$$P(\cdot, \cdot): V \times V \rightarrow [0,+\infty),$$
such that $\sum_{y\in V} P(x,y) = 1 $ $\forall x\in V$. Here
$P (x, y)$ is called reversible if there exists a positive function $\mu (\cdot)$ on the state space $V$, such that $P(x,y) \mu (x) = P(y,x) \mu (y)$.
A random walk is reversible if its underlying Markov kernel is reversible.
It is easy to see that a random walk on a connected   hypergraph with co-rank greater than 2 is \textit{ergodic}, i.e., $P$ is (i) \textit{irreducible}: i.e., for all $x,y \in V$, $P^t (x,y) > 0$ for some $t\in \mathbb{N}$ and (ii) \textit{aperiodic}: i.e., g.c.d $\{t : P^t (x,y) \} = 1$.
We can consider the transition probabilities $P_\G (x,y)$ for  a connected   hypergraph $\G(V,E)$ with $cr(\G)>2$  as \begin{equation*}
P_\G (x,y) = \left\{ \begin{array}{ll}
\frac{\sum_{\substack{e  \in E, \\ x,y\in e} } \frac{1}{|e|-1}}{d_x} & \textrm{if $x \sim y$,}\\
0 & \textrm{else.}
\end{array} \right.
\end{equation*}
Now, let us consider $P_\G : l^2(V, \mu) \rightarrow l^2(V, \mu)$ as a transition probability operator for the random walk on $\G$. Thus $\Delta_{\G} =  \mathbf{I} - P_\G$, where $\mathbf{I}$ is the identical operator in $l^2(V, \mu)$. Hence, for an eigenvalue $\lambda (\Delta_{\G})$ of $\Delta_{\G}$, we always get an eigenvalue $(1 - \lambda (\Delta_{\G}))$ of $P_\G$. Let $\alpha_k = 1- \lambda_k(\Delta_\G)$ be the eigenvalues of $P_\G$ of an $m(>2)$ hypergraph $\G$ on $n$ vertices, for $k=0,\dots,n$. Then, $-1 < \alpha_n \le \alpha_{n-1} \le \alpha_2 \le \alpha_1 =1$. Hence $||P_\G||\le 1$, since the Spec $P_\G \subset (-1,1]$. Let us consider the powers  $P_\G ^t$ of $P_\G$ for $t \in \mathbb{Z}^+$ as composition of operators. Below we recall the theorem for  convergence of random walk on graphs \cite{Grigoryan2011} in the context of  connected   hypergraphs $\G(V,E)$ with $cr(\G)>2$ on $n$ vertices.
\begin{thm}
For any function $f\in  l^2(V, \mu)$, take 
$$\overline{f} = \frac{1}{vol(V)} \sum_{i\in V} d_i f(i) .$$
Then, for any positive integer $t$, we have
$$|| P_\G ^t f - \overline{f}  || \le \rho^t ||f||,$$
where $\rho =\max_{k \ne 1} |1- \lambda_k(\Delta_\G)| = \max (|1- \lambda_2(\Delta_\G)| , |1- \lambda_n(\Delta_\G)| )$ is the spectral radius of $P_\G$ and $||f|| = \sqrt{\langle f,f \rangle_{\mu}}$.
\newline
Consequently,
$$|| P_\G ^t f - \overline{f}  || \rightarrow 0 $$ as $t \rightarrow \infty$, i.e., $P_\G ^tf$ converges to a constant $\overline{f} $ as $t \rightarrow \infty$.
\end{thm}
Thus, after $t \ge \frac{1}{1-\rho} \log (1/\epsilon)$ steps $|| P_\G ^t f - \overline{f} ||$ becomes less than $\epsilon ||f|| $. We define the equilibrium transition probability operator $\overline{ P_\G} : l^2(V, \mu) \rightarrow  l^2(V, \mu)$ as
$$\overline{ P_\G} u(i) = \frac{1}{vol(V)} \sum_{j\in V} d_j u(j)  .$$ Thus, $ \overline{ P_\G} f = \overline{ f} $, for all functions $f\in  l^2(V, \mu)$. Using the above theorem we find that 
$P_\G ^t$ converges to $\overline{ P_\G}$ as $t \rightarrow \infty$.

 We also refer our readers to \cite{Lu2011} where a set of Laplacians for hypergraphs have been defined to study high-order random walks on hypergraphs.

\section{Ricci curvature on hypergraphs}

Here we discuss two aspects of Ricci curvature on hypergraphs. 
Let us recall our transition probability operator 
$$P_\G (x,y) = \frac{\sum_{\substack{e  \in E, \\ x,y\in e} } \frac{1}{|e|-1}}{d_x}$$
for a  hypergraph $\G(V,E)$. Clearly $P_\G (x,y)$ is reversible. Let us define the Laplace operator 
$$\Delta := - \Delta_\G $$
which is also acting on $l^2(V, \mu)$. Thus $\Delta =P_\G- \mathbf{I}$ and for any $f\in l^2(V, \mu)$ we have $(\Delta f)(i) =  \frac{1}{d_i}\sum_{j, i\sim j} \sum_{\substack{e  \in E, \\ i,j\in e} } \frac{1}{|e|-1}(f(j)-f(i))$. Now we discuss two aspects of Ricci curvature in the sense of Bakry and Emery \cite{Bakry1985} and Ollivier \cite{Ollivier2009}. For graphs, readers may  also see \cite{ Bauer2017, Jost2014, Lin2010}.

\subsection{Ricci curvature on hypergraphs in the sense of Bakry and Emery}
Let us define a bilinear operator
$$\Gamma : l^2(V, \mu) \times l^2(V, \mu) \rightarrow l^2(V, \mu)$$
as
$$\Gamma (f_1, f_2)(i) := \frac{1}{2} \{\Delta (f_1(i) f_2(i)) -f_1(i)\Delta f_2(i) -  f_2(i)\Delta f_1(i)  \}. $$ 
Then the Ricci curvature operator, $\Gamma_2$, is defined as
$$\Gamma_2 (f_1, f_2)(i) := \frac{1}{2} \{\Delta \Gamma  (f_1, f_2)(i) - \Gamma  (f_1, \Delta f_2)(i) - \Gamma ( f_2, \Delta f_1)(i)  \}. $$ 
Now, for our hypergraph $\G$ we have
$$\Gamma (f, f)(i) = \frac{1}{2} \frac{1}{d_i} \sum_{j, j\sim i}  \sum_{\substack{e  \in E, \\ i,j\in e} } \frac{1}{|e|-1} (f(i) - f(j))^2.$$
Then, from the proof of Theorem 1.2 in \cite{Lin2010}, we can express our Ricci curvature operator on a hypergraph $\G$ as
\begin{align*}  
    \Gamma_2 (f, f)(i)     = &  \frac{1}{4} \frac{1}{d_i}  \sum_{j, j\sim i} \frac{ 1}{d_j } \sum_{\substack{e  \in E, \\ i,j\in e} } \frac{1}{|e|-1}  \sum_{k, k\sim j}\sum_{\substack{e  \in E, \\ j,k\in e} } \frac{1}{|e|-1} (f(i) - 2f(j) +f(k) )^2 \\
     & { }  -  \frac{1}{2} \frac{1}{d_i}  \sum_{j, j\sim i}  \sum_{\substack{e  \in E, \\ i,j\in e} } \frac{1}{|e|-1} (f(i) - f(j))^2   + \frac{1}{2} \bigg(\frac{1}{d_i}  \sum_{j, j\sim i}   \sum_{\substack{e  \in E, \\ i,j\in e} } \frac{1}{|e|-1}(f(i) - f(j))       \bigg )^2   .
\end{align*}
We have omitted the variable $i$ in the above equation. For simplicity, we do the same for the following equations which hold for all $i\in V$.

Let $\mathbf{m}$ and $ \mathbf{K}$ be the \textit{dimension} and the lower bound of the Ricci  curvature, respectively, of Laplacian operator $\Delta$. Then we say that $\Delta$ satisfies \textit{curvature-dimension type inequality} $CD (\mathbf{m}, \mathbf{K})$ for some $\mathbf{m}>1 $ if
$$\Gamma_2(f,f) \ge \frac{1}{\mathbf{m}} (\Delta f)^2 + \mathbf{K} \Gamma(f,f) \text{ for all } f\in l^2(V, \mu).$$  
If $\Gamma_2 \ge \mathbf{K} \Gamma $, then $\Delta$ satisfies $CD (\infty, \mathbf{K})$. Any connected   (finite) hypergraph $\G(V,E)$ satisfies  $CD (2, \frac{1}{d_*}-1)$, where $d_* = \sup_{i\in V} \sup_{j\sim i} (d_i)/ (\sum_{\substack{e  \in E, \\ i,j\in e} } \frac{1}{|e|-1}) $.

Now, Theorem  2.1 in \cite{Bauer2017} can be stated in the context of hypergraphs as follows.
\begin{thm}
If $\Delta$ satisfies a curvature-dimension type inequality $CD (\mathbf{m}, \mathbf{K})$ with $\mathbf{m}>1 $ and  $\mathbf{K}>0 $ then
$$\lambda_2(\Delta_\G) \ge \frac{ \mathbf{m} \mathbf{K}   }{ \mathbf{m} -1  }.$$
\end{thm}

\subsection{Ricci curvature on hypergraphs in the sense of Ollivier}
The Ollivier's Ricci curvature (also known as  Ricci-Wasserstein curvature) is introduced on a separable and complete metric space $(X,d)$, where each point $x\in X$ has a probability measure $p_x(\cdot)$. Let us denote the structure by $(X,d,p)$. 
Let $\mathcal{C} (\mu, \nu)$ be the set of probability measures on $X \times X$ projecting to $\mu$ and $\nu$. Now $\xi \in \mathcal{C} (\mu, \nu)$ satisfies 
$$ \xi (A \times X) = \mu (A),  \xi (X \times B) = \nu (B), \forall A,B \subset X.$$
Then the \textit{transportation distance (or Wasserstein distance)} between two probability measures $\mu$,  $\nu$ on a metric space $(X,d)$ is defined as 
$$\mathcal{T}_1 (\mu ,\nu)  :=  \inf_{\xi \in \mathcal{C} (\mu, \nu)}  \int_{ X \times X} d(x,y) d\xi(x,y).$$
Now on $(X,d,p)$, the Ricci curvature of $(X,d,p)$ for distinct $x,y\in X$ is defined as 
$$ \kappa (x,y):= 1 - \frac{\mathcal{T}_1 (p_x ,p_y)}{d(x,y)}.$$
For a connected   hypergraph $\G(V,E)$ we take $d(x,y) = 1$ for two distinct adjacent vertices $x,y$   and we consider the probability measure
\begin{equation*} 
p_x(y) = \left\{ \begin{array}{ll}
 \frac{1}{d_x}  \sum_{\substack{e  \in E, \\ x,y\in e} } \frac{1}{|e|-1}, & \textrm{if $y\sim x$,}\\
 0, & \textrm{otherwise,}
  \end{array} \right.
\end{equation*}
for all $x\in V$.
Now  Theorem 3.1 in  \cite{Bauer2017} also holds for a connected  hypergraph as follows
\begin{thm}
Let $\G$ be a connected hypergraph on $n$ vertices. Then $\kappa \le \lambda_2(\Delta_\G) \le \lambda_n(\Delta_\G) \le 2-\kappa$, where the Ollivier's Ricci curvature of $\G$ is at least  $\kappa$.
\end{thm}
As in \cite{Jost2014}, one can also introduce a scalar curvature (suggested in Problem Q in \cite{Ollivier2009}) for a vertex $x$ in $\G$ as
$$
 \kappa (x):= \frac{1}{d_x} \sum_{y, x\sim y} \kappa (x,y).
$$

\section*{Acknowledgements}
The author is sincerely thankful to Richard Buraldi for inspiring him to introduce  connectivity matrix on a hypergraph. The author is very grateful to  Saugata Bandyopadhyay and Asok Nanda for scrutinizing the manuscript. 
The author is also thankful to 
J\"urgen Jost, Shipping Liu,
Satyaki Mazumder, Shibananda Biswas, Sushil Gorai, Shirshendu Chowdhury, Swarnendu Datta and Amitesh Sarkar for fruitful discussions.


\begin{thebibliography}{999}
	
\bibitem{AgarwalEtAl2005} S. Agarwal ; Jongwoo Lim ; L. Zelnik-Manor ; P. Perona ; D. Kriegman ; S. Belongie, Beyond pairwise clustering, In 2005 IEEE Computer Society Conference on Computer Vision and Pattern Recognition (CVPR'05), 2:838-845, 2005.	
	
\bibitem{Banerjee2017} 	A. Banerjee, A. Char, B. Mondal. Spectra of general hypergraphs. 	Linear Algebra and its Applications, 518, 14-30, 2017.
	
	
\bibitem{Banerjee2008} A. Banerjee, J. Jost. On the spectrum of the normalized graph Laplacian. Linear Algebra and its applications, 428:3015-3022, 2008.

\bibitem{Bakry1985} D. Bakry, Michel \'{E}mery. Diffusions hypercontractives (French); S\'{e}minaire de probabilit\'{e}s, XIX, 1983/84, Lecture Notes in Math., vol. 1123, Springer, page 177-206 1985. 

\bibitem{bapat2010} R.B. Bapat. Graphs and Matrices. Springer, 2010.

\bibitem{Bauer2017} F. Bauer, F. Chung, Y. Lin, Y. Liu, Curvature aspects of graphs,   Proceedings of the AMS,   145(5):2033-2042, 2017.

\bibitem{Berge1984} C. Berge, Hypergraphs: Combinatorics of Finite Sets, North-Holland Mathematical Library, Elsevier Science, 1984.

\bibitem{Brouwer2011}  A.E. Brouwer, H.W. Haemers Spectra of graphs. Springer, 2011.

\bibitem{BC1996} G. Burosch, P. V. Ceccherini. A characterization of cube-hypergraphs. Discrete Mathematics 152:55-68, 1996.

\bibitem{Chang2008} K.C. Chang, K. Pearson, T. Zhang. Perrone Frobenius theorem for nonnegative tensors. Communications in Mathematical Sciences, 6(2):507-520, 2008.

\bibitem{Chang2009} K.C. Chang, K. Pearson, T. Zhang. On eigenvalue problems of real symmetric tensors. Journal of Mathematical Analysis and Applications, 350:416-422, 2009.

\bibitem{Cheeger1970} J. Cheeger. A lower bound for the smallest eigenvalue of the Laplace operator. Problems in Analysis (R. C. Gunning, ed.), page 195-199, 1970.
%
%
\bibitem{chung1989} F.R. Chung. Diameters and eigenvalues. Journal of the American Mathematical Society, 2(2):187-196, 1989.


\bibitem{chung1997} F.R. Chung. Spectral graph theory. American Mathematical Society, 1997.

\bibitem{Cooper2012} J. Cooper, A. Dutle. Spectra of uniform hypergraphs. Linear Algebra and its applications, 436:3268-3292, 2012.

\bibitem{Das2008} K.C. Das, R.B.Bapat.  A sharp upper bound on the spectral radius of weighted graphs, Discrete Math., 308(15):3180-3186, 2008.

\bibitem{Cvetkovic2009} D. Cvetkovi\'{c}, P. Rowlinson, S. Simi\'{c}. An Introduction to the Theory of Graph Spectra. Cambridge University Press, 2009.

\bibitem{Ger} S. A. Ger{\v s}gorin, \textit{\"{U}ber die Abgrenzung der Eigenwerte einer Matrix}, Izv. Akad. Nauk. USSR Otd. Fiz.-Mat. Nauk, 6 (1931) 749-754.

\bibitem{Grigoryan2011} Alexander Grigoryan. Analysis on Graphs, Lecture Notes, University Bielefeld, 2011.
%
%
\bibitem{Hu2013_2}  S. Hu, L. Qi, J-Y Shao. Cored hypergraphs, power hypergraphs and their Laplacian H-eigenvalues.  Linear Algebra and its Applications, 439:2980-2998, 2013.

\bibitem{Hu2014} S. Hu, L. Qi. The eigenvectors associated with the zero eigenvalues of the Laplacian and signless Laplacian tensors of a uniform hypergraph, Discrete Applied Mathematics, 169:140-151, 2014.
%
 \bibitem{Hu2015}   S. Hu, L. Qi, J. Xie. The largest Laplacian and signless Laplacian H-eigenvalues of a uniform hypergraph. Linear Algebra and its Applications, 469:1-27, 2015.
%
%
\bibitem{Jost2014} J. Jost, S. Liu. Ollivier’s Ricci Curvature; Local Clustering and Curvature-Dimension Inequalities on Graphs.   Discrete Comput. Geom., 51:300-322, 2014.
 

\bibitem{Li2013} G.Li, L.Qi, G.Yu. The Z-eigenvalues of a symmetric tensor and its application to spectral hypergraph theory. Numerical Linear Algebra with Applications, 20:1001-1029, 2013.


\bibitem{Li2014} J.Li,  J.-M. Guo, W.C.Shiu. Bounds on normalized Laplacian eigenvalues of graphs, Journal of Inequalities and Applications, 2014(1):316, 2014.


\bibitem{Lim2005} L.-H.Lim. Singular values and eigenvalues of tensors: a variational approach. In Computational Advances in Multi-Sensor Adaptive Processing. 2005 1st IEEE International workshop. page 129-132, 2005.


\bibitem{Lin2010} Y.Lin, S.-T. Yau,  Ricci curvature and eigenvalue estimate on locally finite graphs, Math. Res. Lett., 17(2):343-356, 2010.

\bibitem{Lu2011} L.  Lu,  X.  Peng,  High-ordered  random  walks  and  generalized laplacians  on  hypergraphs, in Proceedings  of  the  8th  international conference   on   Algorithms   and   models   for   the   web   graph ,   ser. WAW’11. Berlin,   Heidelberg:   Springer-Verlag,  page  14-25, 2011.

\bibitem{mohar1989} B. Mohar. Isoperimetric numbers of graphs, Journal of Combinatorial Theory, Series B, 47:274-291, 1989.

%
%
\bibitem{mohar1991} B. Mohar. Eigenvalues, diameter, and mean distance in graphs, 
Graphs and Combinatorics 7(1):53-64, 1991.

\bibitem{Ng2009} M. Ng, L. Qi, G. Zhou. Finding the largest eigenvalue of a nonnegetive tensor. SIAM Journal on Matrix Analysis and Applications, 31(3):1090-1099, 2009.

\bibitem{Ollivier2009} Y. Ollivier. Ricci curvature of Markov chains on metric spaces, J. Funct. Anal. 256(3):810-864, 2009.

%

\bibitem{Qi2005}  L. Qi. Eigenvalues of a real supersymmetric tensharpsor. Journal of Symbolic Computation, 40:1302-1324, 2005.

\bibitem{Qi2014} L. Qi. $H^+ $ eigenvalues of laplacian and signless laplacian tensor. Communications in Mathematical Sciences, 12(6):1045-1064, 2014.

\bibitem{Qi2014_2} L. Qi, J. Shao, Q. Wang. Regular uniform hypergraphs, s-cycles, s-paths and their largest Laplacian H-eigenvalues. Linear Algebra and its Applications,  443:215-227, 2014.

\bibitem{Qi2017} L. Qi, Z. Luo. Tensor Analysis: Spectral Theory and Special Tensors, SIAM, 2017.

\bibitem{ROJO2007}sharp O. Rojo. A nontrivial upper bound on the largest Laplacian eigenvalue of weighted graphs, Linear Algebra and its Applications, 420(2):625-633, 2007.


\bibitem{Rodriguez2002} J.A. Rodr\'iguez , On the Laplacian eigenvalues and metric parameters of hypergraphs,  Linear and Multilinear Algebra, 50 (1):1-14, 2002.


\bibitem{Rodriguez2003} J.A. Rodr\'iguez ,  On the Laplacian Spectrum and Walk-regular Hypergraphs, Linear and Multilinear Algebra, 51:3, 285-297,2003.

\bibitem{Rodriguez2009} J.A. Rodr\'iguez , Laplacian eigenvalues and partition problems in hypergraphs, Applied Mathematics Letters, 22(6):916-921,  2009.



\bibitem{Shao2013} J. Shao. A general product of tensors with applications. Linear Algebra and its applications, 439:2350-2366, 2013.

\bibitem{shaoshan2013} J. Shao, H. Shan, L. Zhang. On some properties of the determinants of tensors. Linear Algebra and its applications, 439:3057-3069, 2013.

\bibitem{Voloshin2002} V.I. Voloshin.  Coloring Mixed Hypergraphs -- Theory, Algorithms and Applications. Fields Institute Monographs 17, AMS, 2002.

\bibitem{Voloshin} V.I. Voloshin. Introduction to Graph and Hypergraph Theory, Nova Science Publishers Inc, 2012.

\bibitem{Yang2010} Y. Yang, Q. Yang. Further results for perron-frobenious theorem for nonnegative tensors. SIAM Journal on Matrix Analysis and AppWilf1967lications, 31(5):2517-2530, 2010.

\bibitem{YangYang20101} Y. Yang, Q. Yang. Further results for perron-frobenious theorem for nonnegative tensors II. SIAM Journal on Matrix Analysis and Applications, 32(4):1236-1250, 2011.

\bibitem{Wilf1967} H.S. Wilf. The Eigenvalues of a Graph and Its Chromatic Number, Journal of the London mathematical Society, 42:330-332, 1967.

\bibitem{ZienEtAl1999}  J.Y. Zien ; M.D.F. Schlag ; P.K. Chan, Multilevel spectral hypergraph partitioning with arbitrary vertex sizes, IEEE Transactions on Computer-Aided Design of Integrated Circuits and Systems, 18(9), 1999.

\end{thebibliography}
\end{document}